\documentclass{jocg}

\usepackage{amsmath,amsthm,amssymb}
\usepackage{graphicx,wrapfig}
\usepackage{mathrsfs,upgreek}
\usepackage{booktabs}
\usepackage[all]{hypcap}

\newcommand{\R}{ {\mathbb R} }
\newcommand{\Z}{ {\mathbb Z} }

\newcommand{\F}{ {\mathbb F} }
\newcommand{\va}{ \mathbf{a} }
\newcommand{\vb}{ \mathbf{b} }

\newcommand{\vs}{ \mathbf{s} }
\newcommand{\vt}{ \mathbf{t} }

\newcommand{\vv}{ \mathbf{v} }
\newcommand{\vw}{ \mathbf{w} }
\newcommand{\vx}{ \mathbf{x} }
\newcommand{\vy}{ \mathbf{y} }
\newcommand{\vz}{ \mathbf{z} }
\newcommand{\vzero}{ \mathbf{0} }
\newcommand{\vone}{ \mathbf{1} }
\newcommand{\st}{\mbox{s.t. }}
\newcommand{\dmsfn}{{\sc DMSFN}}

\newcommand{\twolinecelleq}[2]{
\setlength{\abovedisplayskip}{0pt}
\setlength{\belowdisplayskip}{0pt}
\parbox{2.0in}{
\centering
\begin{multline*}
#1\\
#2
\end{multline*}
}
}

\newcommand{\boundary}{\partial}
\newcommand{\abs}[1]{\lvert#1\rvert}
\newcommand{\norm}[1]{\lVert#1\rVert}

\theoremstyle{plain}
\newtheorem{theorem}{Theorem}[section]
\newtheorem{lemma}[theorem]{Lemma}
\theoremstyle{definition}
\newtheorem{definition}[theorem]{Definition}
\theoremstyle{remark}
\newtheorem{remark}[theorem]{Remark}

\newcommand{\spn}{\operatorname{span}}
\newcommand{\spt}{\operatorname{spt}}
\newcommand{\perimeter}{\operatorname{Per}}
\newcommand{\per}{\operatorname{P}}
\newcommand{\diam}{\operatorname{D}}
\newcommand{\intrr}{\operatorname{Int}}
\newcommand{\vol}{\operatorname{V}}
\newcommand{\mass}{\operatorname{M}}
\newcommand{\MSFN}{multiscale simplicial flat norm}
\newcommand{\OHCP}{optimal homologous chain problem}
\newcommand{\OBCP}{optimal bounding chain problem}

\title{\MakeUppercase{Simplicial Flat Norm with Scale}}

\author{Sharif~Ibrahim,\thanks{\affil{Department of Mathematics, Washington State University, Pullman, WA 99164-3113},\newline
\email{\{math.msfn,bkrishna,vixie\}@\{sharifibrahim.com,math.wsu.edu,speakeasy.net\}}}\,\,
  Bala~Krishnamoorthy,\footnotemark[1]\,\,\footnote{Corresponding author}\,\,
  Kevin R.~Vixie\footnotemark[1]
}

\begin{document}

\maketitle

\begin{abstract}
\noindent We study the multiscale simplicial flat norm (MSFN) 
problem, which computes flat norm at various scales of sets defined as
oriented subcomplexes of finite simplicial complexes in arbitrary
dimensions.  We show that the \MSFN{} is NP-complete when homology is defined
over integers. We cast the \MSFN{} as an instance of integer linear
optimization.  Following recent results on related problems, the \MSFN{}
integer program can be solved in polynomial time by solving its linear
programming relaxation, when the simplicial complex satisfies a simple
topological condition (absence of relative torsion). Our most
significant contribution is the {\em simplicial deformation theorem},
which states that one may approximate a general current with a
simplicial current while bounding the expansion of its mass. We
present explicit bounds on the quality of this approximation, which
indicate that the simplicial current gets closer to the original
current as we make the simplicial complex finer. The \MSFN{} opens up the
possibilities of using flat norm to denoise or extract scale
information of large data sets in arbitrary dimensions. On the other
hand, it allows one to employ the large body of algorithmic results on
simplicial complexes to address more general problems related to
currents.
\end{abstract}

\section{Introduction} \label{sec:introduction} 

{\em Currents} are standard objects studied in geometric measure
theory, and are named so by analogy with electrical currents that have
a kind of magnitude and direction at every point. Intuitively, one
could think of currents as generalized surfaces with orientations and
multiplicities. The mathematical machinery of currents has been used
to tackle many fundamental questions in geometric analysis, such as
the ones related to area minimizing surfaces, isoperimetric problems,
and soap-bubble conjectures \cite{Morgan2008}. 

To formally define $d$-currents in $\R^n$, we first let
$\mathcal{D}^d$ be the set of $C^\infty$ differentiable $d$-forms with
compact support.  Then the set of $d$-currents is given by the dual
space of $\mathcal{D}^d$ (denoted $\mathcal{D}_d$) with the weak
topology.  We denote by $\mathcal{R}_m$ the set of rectifiable
currents, which contains all currents that represent oriented
rectifiable sets (i.e., sets which are almost everywhere the countable
union of images of Lipschitz maps from $\R^m$ to $\R^n$) with integer
multiplicities and finite total mass (with multiplicities).

The {\em mass} $\mass(T)$ of a $d$-dimensional current $T$ can be
thought of intuitively as the weighted $d$-dimensional volume of the
generalized object represented by $T$. For instance, the mass of a
$2$-dimensional current can be taken as the area of the surface it
represents.  Formally, the mass of $T$ is given by $M(T) = \sup_{\phi
  \in \mathcal{D}^d} \{T(\phi) \mid \sup \| \phi(x) \| \leq 1\}.$

The boundary $\boundary T$ of a current $T$ is defined by duality with
forms.  That is, we have $\boundary T(\phi) = T(d \phi)$ for every
differential form $\phi \in \mathcal{D}^d$.  Note that when $T$
represents a smooth oriented manifold with boundary, this corresponds
to the usual definition of boundary.  We restrict our attention to
integral currents $T$ that are rectifiable currents with a rectifiable
boundary (i.e., $T \in \mathcal{R}_m$ and $\boundary T \in
\mathcal{R}_{m-1}$).
%have the property that $T \in \mathcal{R}_m$ and $\boundary T \in \mathcal{R}_{m-1}$.
The {\em flat norm} of a $d$-dimensional current $T$ is given by
\begin{equation}
\label{eq:flatnorm}
\F(T) = \min_S \{ \mass(T-\boundary S) + \mass(S)~\big|~ T-\boundary S 
	\in {\mathscr E}_d, \, S \in {\mathscr E}_{d+1}\},
\end{equation}
where ${\mathscr E}_d$ is the set of $d$-dimensional currents with
compact support. One also uses flat norm to measure the ``distance''
between two $d$-currents. More precisely, the flat norm distance
between two $d$-currents $T$ and $P$ is given by
\begin{equation} \label{eq:flatdist}
\F(T,P) = \inf \{\mass(Q) + \mass(R) \, \big| \, T-P = Q+\boundary R,
\, Q \in {\mathscr E}_d, \, R \in {\mathscr E}_{d+1}\}.
\end{equation}
Morgan and Vixie \cite{MoVi2007} showed that the $L^1$ total variation
functional ($L^1$TV) introduced by Chan and Esedo\=glu \cite{ChEs2005}
computes the flat norm for boundaries $T$ with integer
multiplicity. Given this correspondence, and the use of {\em scale} in
$L^1$TV, Morgan and Vixie defined \cite{MoVi2007} the flat norm with
scale $\lambda \in [0,\infty)$ of an oriented $d$-dimensional set $T$
as
\begin{equation}
\label{eq:flatnormwscale}
\F_{\lambda}(T) \equiv \min_S \{ \vol_d(T-\boundary S) + \lambda \vol_{d+1}(S) \},
\end{equation}
where $S$ varies over oriented $(d+1)$-dimensional sets, and $\vol_d$
is the $d$-dimensional volume, used in place of mass. Figure
\ref{fig:1dcurrent} illustrates this definition. Flat norm of the 1D
current $T$ is given by the sum of the length of the resulting
oriented curve $T - \boundary S$ (shown separated from the input curve
for clarity) and the area of the 2D patch $S$ shown in red. Large
values of $\lambda$, above the curvature of both humps in the curve
$T$, preserve both humps. Values of $\lambda$ between the two
curvatures eliminate the hump on the right. Even smaller values
``smooth out'' both humps as illustrated here, giving a more ``flat''
curve, as $S$ can now be comprised of much bigger 2D patches.

%\begin{wrapfigure}{r}{3in} 
\begin{figure}[ht!]  
\centering
\includegraphics[scale=0.7, trim=1.25in 8.3in 2.8in 1in, clip]{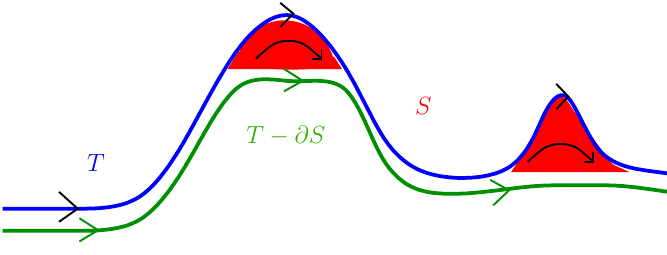}
\caption{\label{fig:1dcurrent} 1D current $T$, and flat norm
decomposition $T-\boundary S$ at appropriate scale $\lambda$. The
resulting current is shown slightly separated from the input current
for clearer visualization.}
\end{figure} %\end{wrapfigure}

Figure \ref{fig:1dcurrent} illustrates the utility of flat norm for
deblurring or smoothing applications, e.g., in 3D terrain maps or 3D
image denoising. But efficient methods for computing flat norm are
known only for certain types of currents in two dimensions. For $d=1$,
Under the setting where $T$ is a boundary, i.e., a loop, embedded in
$\R^2$ and the minimizing surface $S \in \R^2$ as well, the flat norm
could be calculated efficiently, for instance, using graph cut methods
\cite{KoZa2002} -- see the work of Goldfarb and Yin \cite{GoYi2009}
and Vixie et al.~\cite{Vietal2010}, and references therein. Motivated
by applications in image analysis, these approaches usually worked
with a grid representation of the underlying space ($\R^2$). Pixels in
the image readily provide such a representation.

While it is computationally convenient that $L^1$TV minimizers give us
the scaled flat norm for the input images, this approach restricts us
to currents that are boundaries of codimension 1.  Correspondingly,
the calculation of flat norm for $1$-boundaries embedded in higher
dimensional spaces, e.g., $\R^3$, or for input curves that are not
necessarily boundaries has not received much attention so
far. Similarly, flat norm calculations for higher dimensional input
sets have also not been well-studied. Such situations often appear in
practice -- for instance, consider the case of an input set $T$ that
is a curve sitting on a manifold embedded in $\R^3$, with choices for
$S$ restricted to this manifold as well. Further, computational
complexity of calculating flat norm in arbitrary dimensions has not
been studied. But this is not a surprising observation, given the
continuous, rather than combinatorial, setting in which flat norm
computation has been posed so far.

Simplicial complexes that triangulate the input space are often used
as representations of manifolds. Such representations use triangular
or tetrahedral meshes \cite{Edbook2006} as opposed to the uniform
square or cubical grid meshes in $\R^2$ and $\R^3$. Various simplicial
complexes are often used to represent data (in any dimension) that
captures interactions in a broad sense, e.g., the Vietoris--Rips
complex to capture coverage of coordinate-free sensor networks
\cite{SiGh2006,SiGh2007a}. It is natural to consider flat norm
calculations in such settings of simplicial complexes for denoising or
regularizing sets, or for other similar tasks. At the same time,
requiring that the simplicial complex be embedded in high dimensional
space modeled by regular square grids may be cumbersome, and
computationally prohibitive in many cases.

\subsection{\large Our Contributions} \label{ssec:ourconts}

We define a {\em simplicial} flat norm (SFN) for an input set $T$
given as a subcomplex of the finite oriented simplicial complex $K$
triangulating the set, or underlying space $\Omega$. More generally,
$T$ is the simplicial representation of a rectifiable current with
integer multiplicities. The choices of the higher dimensional sets $S$
are restricted to $K$ as well. We extend this definition to the
multiscale simplicial flat norm (MSFN) by including a scale parameter
$\lambda$. The simplicial flat norm is thus a special case of the
\MSFN{} with the default value of $\lambda=1$.

This discrete setting lets us address the worst case complexity of
computing flat norm. Given its combinatorial nature, one would expect
the problem to be difficult in arbitrary dimensions.  Indeed, we show
the problem of computing the \MSFN{} is NP-complete by reducing the
optimal bounding chain problem (OBCP), which was recently shown to be
NP-complete \cite{DuHi2011}, to a special case of the \MSFN{}
problem. We cast the problem of finding the optimal $S$, and thus
calculating the \MSFN{}, as an integer linear programming (ILP)
problem.  Given that the original problem is NP-complete, instances of
this ILP could be hard to solve. Utilizing recent work \cite{DHK2010}
on the related optimal homologous chain problem (OHCP), we provide
conditions on $K$ under which this ILP problem can in fact be solved
in polynomial time. In particular, the \MSFN{} can be computed in
polynomial time when $T$ is $d$-dimensional, and $K$ is
$(d+1)$-dimensional and orientable, for all $d \geq 0$. A similar
result holds for the case when $T$ is $d$-dimensional, and $K$ is
$(d+1)$-dimensional and embedded in $\R^{d+1}$, for all $d \geq 0$.

Our most significant contribution is the {\em simplicial deformation
  theorem} (Theorem \ref{thm:simpldeform}), which states that given an
arbitrary $d$-current in $\abs{K}$ (underlying space), we are assured
of an approximating current in the $d$-skeleton of $K$. This result is
a substantial modification and generalization of the classical
deformation theorem for currents on to square grids. Our deformation
theorem explicitly specifies the dependence of the bounds of
approximation on the regularity and size of the simplices in the
simplicial complex. Hence it is immediate from the theorem that as we
refine the simplicial complex $K$ while preserving the bounds on
simplicial regularity, the flat norm distance between an arbitrary
$d$-current in $\abs{K}$ and its deformation onto the $d$-skeleton of
$K$ vanishes. More importantly, such refinement of $K$ does not affect
the efficient computability of the \MSFN{} by solving the associated
ILP in many cases, e.g., when $K$ is orientable or when it is
full-dimensional.

\subsection{Work on Related Problems} \label{ssec:relprobs}

The problem of computing \MSFN{} is closely related to two other
problems on chains defined on simplicial complexes -- the optimal
homologous chain problem (OHCP) and the optimal bounding chain problem
(OBCP). Given a $d$-chain $\vt$ of the simplicial complex $K$, the
\OHCP{} is to find a $d$-chain $\vx$ that is homologous to $\vt$ such
that $\norm{\vx}_1$ is minimal. In the \OBCP{}, we are given a
$d$-chain $\vt$ of $K$, and the goal is to find a $(d+1)$-chain $\vs$
of $K$ whose boundary is $\vt$ and $\norm{\vs}_1$ is minimal. The
\OBCP{} is closely related to the problem of finding an
area-minimizing surface with a given boundary
\cite{Morgan2008}. Computing the \MSFN{} could be viewed, in a simple
sense, as combining the objectives of the corresponding optimal
homologous chain and optimal bounding chain problem instances, with
the scale factor determining the relative importance of one objective
over the other.

When $\vt$ is a cycle and the homology is defined over $\Z_2$, Chen
and Freedman showed that the \OHCP{} is NP-hard \cite{ChFr2010a}. Dey,
Hirani, and Krishnamoorthy \cite{DHK2010} studied the original version
of the \OHCP{} with homology defined over $\Z$, and showed that the problem
is in fact solvable in polynomial time when $K$ satisfies certain
conditions (when it has no relative torsion).  Recently, Dunfield and
Hirani \cite{DuHi2011} have shown that the \OHCP{} with homology defined
over $\Z$ is NP-complete. We will use their results to show that the
problem of computing the \MSFN{} is NP-complete (see Section
\ref{ssec:npcompMSFN}). These authors also showed that the \OBCP{} with
homology defined over $\Z$ is NP-complete as well. Their result builds
on the previous work of Agol, Hass, and Thurston \cite{AgHaTh2006},
who showed that the knot genus problem is NP-complete, and a slightly
different version of the least area surface problem is NP-hard.

The standard simplicial approximation theorem from algebraic topology
describes how continuous maps are approximated by simplicial maps that
satisfy the star condition \cite[\S14]{Munkres1984}. Our simplicial
deformation theorem applies to currents, which are more general
objects than continuous maps. More importantly, we present explicit
bounds on the expansion of mass of the current resulting from
simplicial approximation. In his PhD thesis, Sullivan
\cite{Sullivan1990} considered deforming currents on to the boundary
of convex sets in a cell complex, which are more general than the
simplices we work with. But simplicial complexes admit efficient
algorithms more naturally than cell complexes. We adopt a different
approach for deformation from Sullivan and obtain new bounds on the
approximations (see Section \ref{ssec:compbnds}). Along with the
\MSFN{}, our deformation theorem also establishes how the \OHCP{} and
\OBCP{} could be used on general continuous inputs by taking
simplicial approximations, thus expanding widely the applicability of
this family of techniques.

\section{Definition of Simplicial Flat Norm} \label{sec:defsfn}

Consider a finite $p$-dimensional simplicial complex $K$ triangulating
the set $\Omega$, where the simplices are oriented, with $p \geq
d+1$. The set $T$ is defined as the integer multiple of an oriented
$d$-dimensional subcomplex of $K$, representing a rectifiable
$d$-current with integer multiplicity. Let $m$ and $n$ be the number
of $d$- and $(d+1)$-dimensional simplices in $K$, respectively. The
set $T$ is then represented by the $d$-chain $\sum_{i=1}^m t_i
\sigma_i$, where $\sigma_i$ are all $d$-simplices in $K$ and $t_i$ are
the corresponding {\em weights}. We will represent this chain by the
vector of weights $\vt \in \Z^m$. We use bold lower case letters to
denote vectors, and the corresponding letter with subscript to denote
components of the vector, e.g., $\vx = [x_j]$. For $\vt$ representing
the set $T$ with integer multiplicity of one, $t_i \in \{-1,0,1\}$
with $-1$ indicating that the orientations of $\sigma_i$ and $T$ are
opposite. But $t_i$ can take any integer value in general. Thus, $\vt$
is the representation of $T$ in the elementary $d$-chain basis of $K$.
We consider $(d+1)$-chains in $K$ modeling sets $S$ representing
rectifiable $(d+1)$-currents with integer multiplicities, and denote
them similarly by $\sum_{j=1}^n s_j \tau_j$ in the elementary
$(d+1)$-chain basis of $K$ consisting of the individual simplices
$\tau_j$. We denote the chain modeling such a set $S$ using the
corresponding vector of weights $\vs \in \Z^n$.

Relationships between the $d$- and $(d+1)$-chains of $K$ are captured
by its $(d+1)$-boundary matrix $[\boundary_{d+1}]$, which is an $m
\times n$ matrix with entries in $\{-1,0,1\}$. If the $d$-simplex
$\sigma_i$ is a face of the $(d+1)$-simplex $\tau_j$, then the $(i,j)$
entry of $[\boundary_{d+1}]$ is nonzero, otherwise it is zero. This
nonzero value is $+1$ if the orientations of $\sigma_i$ and $\tau_j$
agree, and is $-1$ when their orientations are opposite. The $d$-chain
representing the set $T - \boundary_{d+1} S$ is then given as 
\begin{equation*}
\label{eq:xeqTminusbdyS}
\vx = \vt - [\boundary_{d+1}] \vs.
\end{equation*}
Notice that $\vx \in \Z^m$. We define the simplicial flat norm (SFN)
of $T$ represented by the $d$-chain $\vt$ in the $(d+1)$-dimensional
simplicial complex $K$ as
\begin{equation}
\label{eq:SFNdef}
F_S(T) = \min_{\vs \in \Z^n} \left\{ \sum_{i=1}^m \vol_d(\sigma_i)
\abs{x_i} + \sum_{j=1}^n \vol_{d+1}(\tau_j) \abs{s_j} ~\big|~ \vx = \vt -
[\boundary_{d+1}] \vs,\, \vx \in \Z^m \right\}.
\end{equation}
Since $\vx$ and $\vs$ are chains in a simplicial complex, the masses
of the currents they represent (as given in
Equation~\ref{eq:flatnorm}) are indeed given by the weighted sums of
the volumes of the corresponding simplices. The integer restrictions
$\vx \in \Z^m$ and $\vs \in \Z^n$ are important in this definition as
we are studying currents with integer multiplicities.  The simplicial
flat norm is intuitively the problem of deforming an input chain to
another chain of least cost, where cost is determined both by the mass
of the resulting chain and the size of the deformation (constrained to
the complex) used to get it.  For instance, in a triangulation of a
manifold, we constrain ourselves to only use deformations on the
manifold.  We generalize the definition of SFN to define a {\em
  multiscale} simplicial flat norm (MSFN) of $T$ in the simplicial
complex $K$ by including a scale parameter $\lambda \in [0,\infty)$.
\begin{equation}
\label{eq:MSFNdef}
F^{\lambda}_S(T) = \min_{\vs \in \Z^n} \left\{ \sum_{i=1}^m
\vol_d(\sigma_i) \abs{x_i} + \lambda \left( \sum_{j=1}^n \vol_{d+1}(\tau_j)
\abs{s_j} \right) ~\big|~ \vx = \vt - [\boundary_{d+1}] \vs,\, \vx
\in \Z^m \right\}.
\end{equation}
This definition is the simplicial version of the multiscale flat norm
defined in Equation (\ref{eq:flatnormwscale}). The default, or
nonscale, simplicial flat norm in Equation~(\ref{eq:SFNdef}) is a
special case of the multiscale simplicial flat norm with the default
value of $\lambda=1$.

The (non-simplicial) flat norm with scale $\lambda > 0$ of a
$d$-dimensional current $T$ can be rewritten as $\F_{\lambda}(T) =
\lambda^d \cdot \F_{1}(T/\lambda)$.  Thus the flat norm with scale can
be thought of as the traditional flat norm applied to a scaled copy of
the input current.  An equivalent statement can be made for the
simplicial flat norm, but crucially requires that the simplicial
complex be similarly scaled.  To avoid this complex scaling issue
especially when considering all possible scales, and to simplify our
notation, we henceforth study the more general multiscale simplicial
flat norm (which also allows us to consider the $\lambda = 0$ case).

We assume the $d$- and $(d+1)$-dimensional volumes of simplices to be
any nonnegative values. For example, when $\sigma_i$ is a $1$-simplex,
i.e., edge, $\vol_1(\sigma_i)$ could be taken as its Euclidean
length. Similarly, $\vol_2(\tau_j)$ for a triangle $\tau_j$ could be
its area. For ease of notation, we denote $\vol_d(\sigma_i)$ by $w_i$
and $\vol_{d+1}(\tau_j)$ by $v_j$, with the dimensions $d$ and $d+1$
evident from the context.

\begin{remark}
\label{rem:MSFNmin}
The minimum in the definition of the \MSFN{} (Equation~\ref{eq:MSFNdef})
indeed exists. The function 
\begin{equation}
\label{eq:funcmsfn}
f^{\lambda}(T,S) = \sum_{i=1}^m w_i \abs{x_i} + \lambda \, (
\sum_{j=1}^n v_j \abs{s_j} )~~\,\mbox{ with }~~\,\vx = \vt -
[\boundary_{d+1}] \vs\,
\end{equation}
is lower bounded by zero, as it is the sum of nonnegative entries (we
have $\lambda \geq 0$). Notice that $F^{\lambda}_S(T) = \min_{S}
f^{\lambda}(T,S)$. Further, we only consider integral $\vs$ defined on
the finite simplicial complex $K$, and hence there are only a finite
number of values for this function. Hence its minimum indeed exists,
which defines the \MSFN{} of $\vt$. On the other hand, the proof of
existence of minimum in the original definition of flat norm for
rectifiable currents employs the Hahn--Banach theorem
\cite[pg.~367]{Federer1969}.
\end{remark}

We illustrate the optimal decompositions to compute the \MSFN{} for
two different scales ($\lambda=1$ and $\lambda \ll 1$) in Figure
\ref{fig:curveonmesh}. Notice that the input set $T$, shown in blue,
is not a closed loop here. It is a subcomplex of the simplicial
complex triangulating $\Omega$. The underlying set $\Omega$ need not
be embedded in $\R^2$ -- it could be sitting in $\R^3$ or any higher
dimension. We do not show the orientations of individual simplices and
chains so as not to clutter the figure. We could take each triangle to
be oriented counterclockwise (CCW), with $T$ oriented CCW as well, and
each edge oriented arbitrarily. When scale $\lambda=1$, we get the
default SFN of $T$, where the $S$ chosen (shown in light pink) is such
that the resulting optimal $T - \boundary S$ (indicated by the thin
curve in dark green) is devoid of all the ``kinks'', but is similar to
$T$ in overall form.  This removal of the tightest ``kinks'' is a
discrete analogue of how the $\lambda$ in the flat norm relates to the
curvature in the continuous case.  For $\lambda \ll 1$, the second
term in the definition (Equation~\ref{eq:MSFNdef}) contributes much
less to the \MSFN{}. As such, the optimal $T- \boundary S$ consists of
a short chain of two edges (shown in light green), which closes the
original $T$ curve to form a loop. $S$ in this case includes the
triangles in the former choice of $S$, and all other triangles
enclosed by the original curve $T$ and the resulting $T -\boundary S$.

\begin{figure}[ht!]
  \centering \includegraphics[scale=0.6, trim=1in 5.5in 0.3in 0.8in,
  clip] {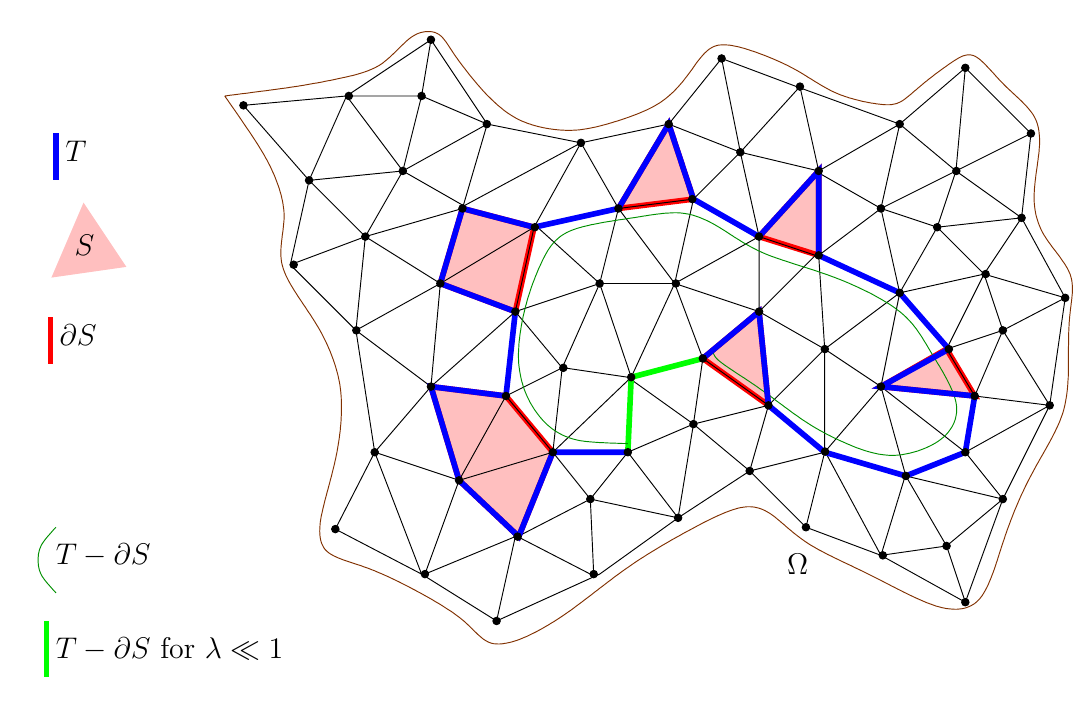}
% \centering \includegraphics[scale=0.6]{Fig_CurveOnMesh.eps}
%
  \caption{\label{fig:curveonmesh} The \MSFN{} illustrated for two different
  scales $\lambda = 1$ and $\lambda \ll 1$. See text for explanation.}
\end{figure}

\subsection{Complexity of \MSFN{}} \label{ssec:npcompMSFN}

To study the complexity of computing the \MSFN{}, we consider a decision
version of the problem, termed {\em decision}-MSFN or \dmsfn. The
function $f^{\lambda}(T,S)$ used here is defined in Equation
\ref{eq:funcmsfn}, with the modification that $w_i$ and $v_j$ are
assumed to be rational for purposes of analyses of complexity.
\begin{definition}[{\rm \dmsfn}] 
  Given a $p$-dimensional finite simplicial complex $K$ with $p \geq
  d+1$, a set $T$ defined as a $d$-subcomplex of $K$, a scale $\lambda
  \in [0,\infty)$, and a rational number $f_0 \geq 0$, does there
    exist a $(d+1)$-dimensional subcomplex $S$ of $K$ such that
    $f^{\lambda}(T,S) \leq f_0$?
\end{definition}
The related optimal homologous chain problem (OHCP) was recently shown
to be NP-complete \cite[Theorem 1.4]{DuHi2011}. We reduce OHCP to a
special case of \dmsfn, thus showing that \dmsfn \ is NP-complete as
well. The default optimization version of MSFN consequently turns out
to be NP-hard.
\begin{theorem} \label{thm:dmsfnnpc}
  \dmsfn \ is NP-complete, and {\rm MSFN} is NP-hard.
\end{theorem}
\begin{proof}
  \dmsfn \ lies in NP as we can calculate $f^{\lambda}(T,S)$ in
  polynomial time when given a pair of $d$- and $(d+1)$-chains $\vt$
  and $\vs$, respectively, of the simplicial complex $K$.  On the
  other hand, given an instance of the optimal homologous chain
  decision problem, we can reduce it to the \dmsfn{} by taking
  $\lambda = 0$ and $w_i = 1$ for $1 \leq i \leq m$.  Since the
  \OHCP{} was recently shown to be NP-complete~\cite[Theorem
    1.4]{DuHi2011}, the result follows.
\end{proof}

\begin{remark}
  Although we showed MSFN is NP-hard in general, the case for any
  particular $\lambda > 0$ is not known.  For $\lambda$ large enough,
  the problem in fact becomes easy-- when the $(d+1)$-simplices have
  positive volumes and $\lambda > (\sum w_i) / \min v_j$, then
  optimality occurs when $\vs$ is the empty $(d+1)$-chain.
\end{remark}

We now consider attacking the \MSFN{} problem using techniques from
the area of discrete optimization. Even though the problem is NP-hard,
this approach helps us to identify special cases in which we can
compute the \MSFN{} in polynomial time.

\section{Multiscale Simplicial Flat Norm and Integer Linear Programming} \label{ssec:MSFNILP}

The problem of finding the \MSFN{} of the $d$-chain $\vt$
(Equation~\ref{eq:MSFNdef}) can be cast formally as the following
optimization problem.
\begin{equation}
\label{eq:optprobMSFN}
\begin{array}{ll}
\mbox{minimize} & \sum_{i=1}^m w_i \abs{x_i} + \lambda
     ( \sum_{j=1}^n v_j \abs{s_j} ) \\
\vspace*{-0.13in} \\
\mbox{subject to} & ~~~~~~\vx = \vt - [\boundary_{d+1}] \vs, \\
 & ~~~~~~\vx \in \Z^m,~ \vs \in \Z^n.
\end{array}
\end{equation}
The objective function is piecewise linear in the integer variables
$\vx$ and $\vs$. Using standard modeling techniques from linear
optimization \cite[pg.~18]{BeTs1997}, we can reformulate the problem
as the following integer {\em linear} program (ILP).
\begin{equation}
\label{eq:ILPMSFN}
\begin{array}{ll}
\min & \sum_{i=1}^m w_i (x^+_i + x^-_i) \,+ \,\lambda
     \left( \sum_{j=1}^n v_j (s^+_j + s^-_j) \right) \\
\vspace*{-0.13in} \\
\st & \vx^+ - \vx^-  = \vt - [\boundary_{d+1}] (\vs^+ - \vs^-) \\
 & \vx^+, \vx^- \geq \vzero,~~\vs^+, \vs^- \geq \vzero \\
 & \vx^+, \vx^- \in \Z^m,~~ \vs^+, \vs^- \in \Z^n.
\end{array}
\end{equation}
The objective function coefficients need to be nonnegative for this
formulation to work -- indeed, we have $w_i, v_j$, and $\lambda$
nonnegative. Integer linear programming is NP-complete
\cite{Schrijver1986}. The linear programming relaxation of the ILP 
above is obtained by ignoring the integer restrictions on the
variables.
\begin{equation}
  \label{eq:LPMSFN}
  \begin{array}{ll}
    \min & \sum_{i=1}^m w_i (x^+_i + x^-_i) \,+ \,\lambda
    \left( \sum_{j=1}^n v_j (s^+_j + s^-_j) \right) \\
    \vspace*{-0.13in} \\
    \st & \vx^+ - \vx^-  = \vt - [\boundary_{d+1}] (\vs^+ - \vs^-) \\
        & \vx^+, \vx^- \geq \vzero,~~\vs^+, \vs^- \geq \vzero 
  \end{array}
\end{equation}

We are interested in instances of this linear program (LP) that have
integer optimal solutions, which hence are optimal solutions for the
original ILP (Equation~\ref{eq:ILPMSFN}) as well. Totally unimodular
matrices yield a prime class of linear programming problems with
integral solutions.  Recall that a matrix is totally unimodular if all
its subdeterminants equal $-1,0,$ or $1$; in particular, each entry is
$-1,0,$ or $1$. The connection between total unimodularity and linear
programming is specified by the following theorem.
\begin{theorem}
\label{thm:IPandTU}
{\rm \cite{VeDa1968}} Let $A$ be an $m \times n$ totally unimodular
matrix, and $\vb \in \Z^m$. Then the polyhedron ${\cal P} = \{ \vx \in
\R^n \, | \, A \vx = \vb, \, \vx \geq \vzero\}$ has integral vertices.
\end{theorem}

Notice that the feasible set of the \MSFN{} LP (Equation~\ref{eq:LPMSFN})
has the form specified in the theorem above, with the variable vector
$(\vx^+, \vx^-, \vs^+, \vs^-)$ in place of $\vx$. The corresponding
equality constraint matrix $A$ has the form $\begin{bmatrix} I & -I &
B & -B \end{bmatrix}$, where $I$ is the identity matrix and $B =
[\boundary_{d+1}]$. The input $d$-chain $\vt$ is in place of the
right-hand side vector $\vb$. In order to use Theorem
\ref{thm:IPandTU} for computing the \MSFN{}, we connect the total
unimodularity of constraint matrix $A$ and that of boundary matrix
$B$.

\begin{lemma}
\label{lem:TUconmat}
If $B = [\boundary_{d+1}]$ is totally unimodular, then so is the
matrix $A = \begin{bmatrix} I & -I & B & -B \end{bmatrix}$.
\end{lemma}

\begin{proof}
Starting with $B$, we get the matrix $A$ by appending columns of $B$
scaled by $-1$ to its right, and appending columns with a single
nonzero entry of $\pm 1$ to its left. Both these classes of operations
preserve total unimodularity \cite[pg.~280]{Schrijver1986}.
\end{proof}

\noindent Consequently, we get the following result on polynomial time
computability of the \MSFN{}.

\begin{theorem}
\label{thm:MSFNpolytime}
If the boundary matrix $[\boundary_{d+1}]$ of the finite oriented
simplicial complex $K$ is totally unimodular, then the multiscale
simplicial flat norm of the set $T$ specified as a $d$-chain $\vt \in
\Z^m$ of $K$ can be computed in polynomial time.
\end{theorem}

\begin{proof}
The problem of computing the \MSFN{} of $T$ (Equation~\ref{eq:MSFNdef})
is cast as the optimization problem given in
Equation~(\ref{eq:optprobMSFN}). This problem is reformulated as an
instance of ILP (Equation~\ref{eq:ILPMSFN}). We get the \MSFN{} LP
(Equation~\ref{eq:LPMSFN}) by relaxing the integrality constraints of
this ILP.  As noted in Remark \ref{rem:MSFNmin}, the optimal cost of
this LP is finite. The polyhedron of this LP has at least one vertex,
given that all variables are nonnegative \cite[Cor.~2.2]{BeTs1997}. 
By Lemma \ref{lem:TUconmat}, the constraint matrix of this LP is
totally unimodular, as $[\boundary_{d+1}]$ is so. Hence by
Theorem~\ref{thm:IPandTU}, all vertices of the feasible region of the \MSFN{}
LP are integral, since $\vt \in \Z^m$.

An optimal solution $( \vx^+_*, \vx^-_*, \vs^+_*, \vs^-_* )$ of the
\MSFN{} LP can be found in polynomial time using an interior point
method \cite[Chap.~9]{BeTs1997}. If it happens to be a unique optimal
solution, then it will be a vertex, and hence will be integral by
Theorem \ref{thm:IPandTU}. Hence it is an optimal solution to the ILP
(Equation~\ref{eq:ILPMSFN}).

If the optimal solution is not unique, then $( \vx^+_*, \vx^-_*,
\vs^+_*, \vs^-_* )$ may be nonintegral. But since the optimal cost is
finite, there must exist a vertex in its polyhedron that has this
minimum cost. Given a nonintegral optimal solution obtained by an
interior point method, one can find such an integral optimal solution
at a vertex in polynomial time \cite{GuHeRoTeTs1993}. Hence the \MSFN{}
ILP can be solved in polynomial time in this case as well.
\end{proof}

\begin{remark}
\label{rem:strgpoly}
We point out that since the boundary matrix $B = [\boundary_{d+1}]$
has entries only in $\{-1,0,1\}$, the constraint matrix of the \MSFN{} LP
(Equation~\ref{eq:LPMSFN}) also has entries only in
$\{-1,0,1\}$. Hence the \MSFN{} LP can be solved in strongly polynomial
time \cite{Tardos1986}, i.e., the time complexity is independent of
the objective function and right-hand side coefficients, and depends
only on the dimensions of the problem.
\end{remark}

\begin{remark}
\label{rem:zeroonesoln}

Components of variables $\vx^+,\vx^-,\vs^+,\vs^-$ in the \MSFN{} ILP
(Equation~\ref{eq:ILPMSFN}) could assume values other than
$\{-1,0,1\}$, indicating integer multiplicities higher than $1$ for
the corresponding simplices in the optimal decomposition. The
definition of \MSFN{} (Equation~\ref{eq:MSFNdef}) does allow such
larger multiplicities. At the same time, if one insists on using each
$(d+1)$-simplex at most {\em once} when calculating the \MSFN{}, and
insists on similar restrictions on $d$-simplices in the optimal
decomposition, we can modify the ILP such that Theorem
\ref{thm:MSFNpolytime} still holds.

Denoting the entire variable vector by $\vx = (\vx^+, \vx^-, \vs^+,
\vs^-) \in \Z^{2m+2n}$, we add the upper bound constraints $\vx \leq
\vone$, where $\vone$ is the $(2m+2n)$-vector of ones. These
inequalities could be converted to the set of equations $\vx + \vy =
\vone$, where $\vy$ is the $(2m+2n)$-vector of slack variables that
are nonnegative. These modifications give an ILP whose polyhedron is
in the same form as described in Theorem \ref{thm:IPandTU}, with the
equations denoted as $A'\vx' = \vb'$ for the variable vector $\vx' =
(\vx,\vy)$. The new constraint matrix $A'$ is related to the
constraint matrix $A$ of the original \MSFN{} ILP given in Lemma
\ref{lem:TUconmat} as 
\[
A' = \begin{bmatrix} A & O \\ I & I \end{bmatrix},
\]
where $I$ is the $2m+2n$ identity matrix, and $O$ is the $m \times
(2m+2n)$ zero matrix. Hence $A'$ is obtained from $A$ by first adding
$2m+2n$ rows with a single nonzero entry of $+1$, and then adding to
the resulting matrix $2n+2m$ more columns with a single nonzero entry
of $+1$. These operations preserve total unimodularity
\cite[pg.~280]{Schrijver1986}, and hence the new constraint matrix
$A'$ is totally unimodular when $[\boundary_{d+1}]$ is so. The new
right-hand side vector $\vb' \in \Z^{3m+2n}$ consists of the input
chain $\vt$ and the vector of ones from the new upper bound
constraints.

\end{remark}

Since the efficient computability of the \MSFN{} depends on the total
unimodularity of the boundary matrix, we study the conditions under
which total unimodularity of boundary matrices can be guaranteed.

\section{Simplicial Complexes and Relative Torsion}

Dey, Hirani, and Krishnamoorthy \cite{DHK2010} have given a simple
characterization of the simplicial complex whose boundary matrix is
totally unimodular. In short, if the simplicial complex does not have
relative torsion then its boundary matrix is totally unimodular. We
state this and other related results here for the sake of
completeness, and refer the reader to the original paper
\cite{DHK2010} for details and proofs. The simplicial complex $K$ in
these results has dimension $d+1$ or higher. Recall that a
$d$-dimensional simplicial complex is {\em pure} \ if it consists of
$d$-simplices and their faces, i.e., there are no lower dimensional
simplices that are not part of some $d$-simplex in the complex.

\begin{theorem} \label{thm:TUreltorfree}
  {\rm \cite[Theorem 5.2]{DHK2010}} The boundary matrix
  $[\boundary_{d+1}]$ of a finite simplicial complex $K$ is totally
  unimodular if and only if $H_d(L,L_0)$ is torsion-free for all pure
  subcomplexes $L_0,L$ of $K$, with $L_0 \subset L$.
\end{theorem}

These authors further describe situations in which the absence of
relative torsion is guaranteed. The following two special cases
describe simplicial complexes for which the boundary matrix is always
totally unimodular.

\begin{theorem} \label{thm:TUorimfld}
  {\rm \cite[Theorem 4.1]{DHK2010}} The boundary matrix
  $[\boundary_{d+1}]$ of a finite simplicial complex triangulating a
  compact orientable $(d+1)$-dimensional manifold is totally
  unimodular.
\end{theorem}

\begin{theorem} \label{thm:TUembedinRdp1}
  {\rm \cite[Theorem 5.7]{DHK2010}} The boundary matrix
  $[\boundary_{d+1}]$ of a finite simplicial complex embedded in
  $\R^{d+1}$ is totally unimodular.
\end{theorem}

\noindent For simplicial complexes of dimension $2$ or lower, the
boundary matrix is totally unimodular when the complex does not have a
M\"obius subcomplex.

\begin{theorem} \label{thm:TUnomobius}
  {\rm \cite[Theorem 5.13]{DHK2010}} For $d \leq 1$, the boundary
  matrix $[\boundary_{d+1}]$ is totally unimodular if and only if the
  finite simplicial complex has no $(d+1)$-dimensional M\"obius
  subcomplex.
\end{theorem}

It is appropriate to mention here that the connection between total
unimodularity of boundary matrices and torsion in the complex has been
observed as early as in 1895 by
Poincar\'{e}\cite{Poincare2010}. However, the result in~\cite{DHK2010}
connecting the total unimodularity with relative torsion is different
and has led to a polynomial time algorithm for the OHCP
problem. Notice that a complex can be torsion-free, but have
non-trivial relative torsion. The M\"{o}bius strip is such an example.

We illustrate the implications of the results above for the efficient
computation of the \MSFN{} by considering certain sets. When the input
set $T$ is of dimension 1, and is described on an orientable
$2$-manifold to which the choices of $2$-dimensional set $S$ are also
restricted, we can always compute its \MSFN{} by solving the \MSFN{}
LP (Equation~\ref{eq:LPMSFN}) in polynomial time.  A similar result
holds when $T$ is a set of dimension $2$ described as a subcomplex of
a $3$-complex sitting in $\R^3$. For a $1$-dimensional set $T$ with
choices of $S$ restricted to a $2$-complex $K$, we can always compute
the \MSFN{} of $T$ efficiently as long as $K$ does not have a
$2$-dimensional M\"obius subcomplex. Notice that $K$ itself need not
be embedded in $\R^3$ for this result to work -- it could be sitting
in some higher dimensional space.

\section{Simplicial Deformation Theorem} \label{sec:simpdefthm}

When can we use the multiscale simplicial flat norm as a discrete
surrogate for the traditional flat norm?  That is, if we wish to solve
a flat norm problem (for which there are no practical algorithms in
general), can we discretize the problem and find a problem close
enough to the original one which we can solve?

The deformation theorem \cite[Sections 4.2.7--9]{Federer1969} is one
of the fundamental results of geometric measure theory, and more
particularly of the theory of currents. It approximates an integral
current by deforming it onto a cubical grid of appropriate mesh
size. On the other hand, we have been studying currents or sets in the
setting of simplicial complexes, rather than on square grids. Our
proof is a substantial modification of the classical proof of the
deformation theorem. We found the presentation of the latter proof by
Krantz and Parks \cite[Section 7.7]{KrantzParks2008} especially
helpful. Our proof mimics their proof when possible. The gist of this
theorem is the assertion that we may approximate a current with a
simplicial current.

Recall that $\vol_d(\sigma)$ denotes the $d$-dimensional volume of a
$d$-simplex $\sigma$. The {\em perimeter} of $\sigma$ is the set of
all its $(d-1)$-dimensional faces, denoted as $\perimeter(\sigma) = \{
\cup_j \tau_j \,|\, \tau_j \in \sigma, \dim(\tau_j) = d-1\}$. We will
also refer to the $(d-1)$-dimensional volume of $\perimeter(\sigma)$
as the perimeter of $\sigma$, but denote it as $\per(\sigma) =
\sum_{\tau_j \in \perimeter(\sigma)} \vol_{d-1}(\tau_j)$. We let
$\diam(\sigma)$ be the diameter of $\sigma$, which is the largest
Euclidean distance between any two points in $\sigma$.

\begin{theorem}[{\bf Simplicial Deformation Theorem}] \label{thm:simpldeform}
  Let $K$ be a $p$-dimensional simplicial complex embedded in $\R^q$,
  with $p=d+k$ for $k \geq 1$ and $q \geq p$. Suppose that for every
  simplex $\sigma \in K$
  \[ \frac{ \diam(\sigma) \per(\sigma)} {\vol_d({\cal B_\sigma})} \leq
  \upkappa_1 < \infty,\]
  \[ \frac{\diam(\sigma)}{r_{\sigma}} \leq \upkappa_2 < \infty,\] and
  \[\diam(\sigma) \leq \Updelta\] 
  hold, where $\hat{\cal B}_\sigma$ is the largest ball inscribed in
  $\sigma$, ${\cal B_\sigma}$ is the ball with half the radius and
  same center as $\hat{\cal B}_\sigma$, and $r_{\sigma}$ is the radius
  of ${\cal B_\sigma}$. Let $T$ be a $d$-dimensional current in $\R^q$
  such that the support of $T$ is a subset of the underlying
  space of $K$. Suppose that $T$ satisfies
  \[ \mass(T) + \mass(\boundary T) < \infty.\] 
  Then there exists a simplicial $d$-current $P$ supported in the
  $d$-skeleton of $K$ whose boundary $\boundary P$ is supported in the
  $(d-1)$-skeleton of $K$ such that
  \[ T-P = Q + \boundary R,\] 
  and the following controls on mass $M$ hold: 
  \begin{align}
    \mass(P) & \leq (4\upvartheta_K)^{k} \mass(T)+ \Updelta(4\upvartheta_{K})^{k+1}\mass(\boundary T),
    \label{eq:simpdefthmMP}\\
    \mass(\boundary P) & \leq (4\upvartheta_K)^{k+1} \mass(\boundary T), 
    \label{eq:simpdefthmMbdyP}\\
    \mass(R) & \leq \Updelta(4\upvartheta_K)^{k} \mass(T), \mbox{ and}
    \label{eq:simpdefthmMR} \\
    \mass(Q) & \leq  \Updelta(4\upvartheta_{K})^{k}(1+4\upvartheta_{K})\mass(\boundary T), 
    \label{eq:simpdefthmMQ}
  \end{align}
  where $\upvartheta_K = \upkappa_1 + \upkappa_2$.
\end{theorem}

\begin{remark}
\label{rem:defthm1}
  It is immediate that the flat norm distance between $T$ and $P$ can
  be made arbitrarily small by subdividing the simplicial complex to
  reduce $\Updelta$ while preserving the regularity of the refinement
  as measured by $\upkappa_1$ and $\upkappa_2$.
\end{remark}
\begin{remark}
  Note that this theorem combines the unscaled and scaled versions of
  the original deformation theorem \cite[Theorems 7.7.1 and
  7.7.2]{KrantzParks2008} into one theorem through the explicit form
  of the constraints. In our proof of Theorem \ref{thm:simpldeform},
  we replace certain pieces of the original proof as presented by
  Krantz and Parks~\cite[Pages 211--222]{KrantzParks2008} without
  reproducing all the other details of their proof. We found their
  exposition quite well-structured, making it easier to identify the
  modifications needed to get our theorem.
\end{remark}
\begin{remark}
  The bound for $\mass(P)$ in Theorem \ref{thm:simpldeform} is larger
  than the classical bound. We get this large bound because we
  generate $P$ through retractions alone, and not using the usual
  Sobolev-type estimates \cite[Pages 220--222]{KrantzParks2008}. And
  of course, the $\Updelta$ in the coefficient of the extra term means
  that it becomes unimportant as the simplicial complex is
  appropriately subdivided.
\end{remark}

\subsection{Proof of the Simplicial Deformation Theorem}

At the heart of the modification of the deformation theorem (from
cubical grid to simplicial complex settings) is the recalculation of
an integral over the current and its boundary. This integral appears
in a bound on the {\em Jacobian of the retraction}, which measures the
expansion in mass of the current resulting from the process of
retracting it on to the simplices of the simplicial complex. To do
this recalculation, we consider the retraction $\phi$ one step at a
time, building it through independent choices of centers to project
from in every simplex and its every face.

We first describe the general set up of retraction within a
simplex. We then present certain bounds on the mass expansion
resulting from the retraction in Lemmas \ref{lem:bnddjacob},
\ref{lem:bndJxfixa}, and \ref{lem:bndJafixx}. In particular, we obtain
bounds on the expansion that are independent of the choice of points
from which we project. These bounds are independent of the particular
current that we retract on to the simplicial complex. But we employ
these bounds to subsequently bound the overall expansion of mass of
the current resulting from the retraction.

\subsubsection{Retracting from a center inside a simplex} \label{sssec:retrsetup}

  We describe the details of retraction for an $\ell$-simplex $\sigma$
  in the $p$-dimensional simplicial complex $K$. This set up is valid
  for any $\ell$, but in particular, we will use the bounds thus
  obtained for $d \leq \ell \leq p$ when retracting a $d$-current onto
  $K$. We pick a center $\va \in \intrr(\sigma)$, the interior of
  $\sigma$, and project every $\vx \in \intrr(\sigma) \setminus
  \{\va\} $ along the ray $(\vx-\va)/\norm{\vx-\va}$ to
  $\perimeter(\sigma)$. Denoting this map as $\phi(\vx,\va)$, we get
  \begin{equation} \label{eq:phixa}
  \phi(\vx,\va) = (\phi_{\pi} \circ \phi_{\delta})(\vx,\va),
  \end{equation}

  where $\phi_{\delta}(\vx,\va)$ is a dilation of $\R^\ell$ by the
  factor $\norm{\phi(\vx,\va)-\va}/\norm{\vx-\va}$ and
  $\phi_{\pi}(\vx, \va)$ is a nonorthogonal projection along $(\vx -
  \va)/ \norm{\vx-\va}$ onto $\tau_{\vx}$, the $(\ell-1)$-dimensional
  face of $\sigma$ containing $\phi(\vx,\va)$. We denote $\hat{r} =
  \norm{\phi(\vx,\va) - \va}$ and $r = \norm{\vx-\va}$. Let $E_{\ell}$
  be the $\ell$-hyperplane that contains $\sigma$ and $E_{\ell-1}$ the
  $(\ell-1)$-hyperplane that contains $\tau_{\vx}$. Denote the
  orthogonal projection of $\va$ onto $E_{\ell-1}$ by $\vb$, and let
  $\hat{h} = \norm{\vb - \va}$. For any point $\vy = \va + (\vb - \va)
  \gamma$ with $0 < \gamma < 1$, we get $\phi(\vy,\va) = \vb$. In
  particular, we consider the point of intersection of line connecting
  $\va$ and $\vb$ with the $(\ell-1)$-hyperplane parallel to
  $\tau_{\vx}$ that contains $\vx$. Naming this point $\vy$, we define
  $h = \norm{\vy - \va}$. Let $\vz \in E_{\ell}$ denote either normal
  to $\tau_{\vx}$ at $\phi(\vx,\va)$ (either of the two possibilities
  work). Let $\vv_2 = (\vx - \va)/ \norm{\vx-\va}$, and let $\vv_1$ be
  the vector in $\spn ( \vz, \vv_2 )$ that is normal to $\vv_2$ and
  points into $\sigma$. We illustrate this construction on a
  $3$-simplex in Figure~\ref{fig:jacob}, where the cone of $\va$ with
  face $\tau$ is shown in red and the other points and vectors are
  labeled. We also illustrate the corresponding slice spanned by
  $\vv_1$ and $\vv_2$ in Figure~\ref{fig:2d-diagram}.
\begin{figure}[htp!]
  \centering
  \includegraphics[scale=1, trim=1in 7.2in 3in 1in, clip]{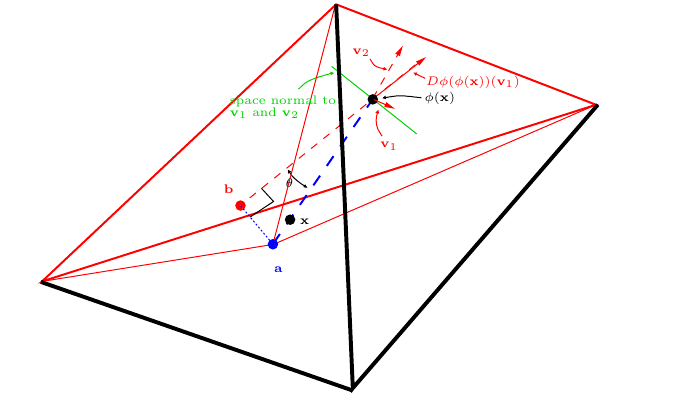}
  \caption{Illustration of the dilation and nonorthogonal projection
  involved in retraction for a $3$-simplex.}  
  \label{fig:jacob}
  \end{figure}
  \begin{figure}[ht!]
  \centering
  \includegraphics[scale=0.9, trim=1in 6.9in 2.2in 1in, clip]{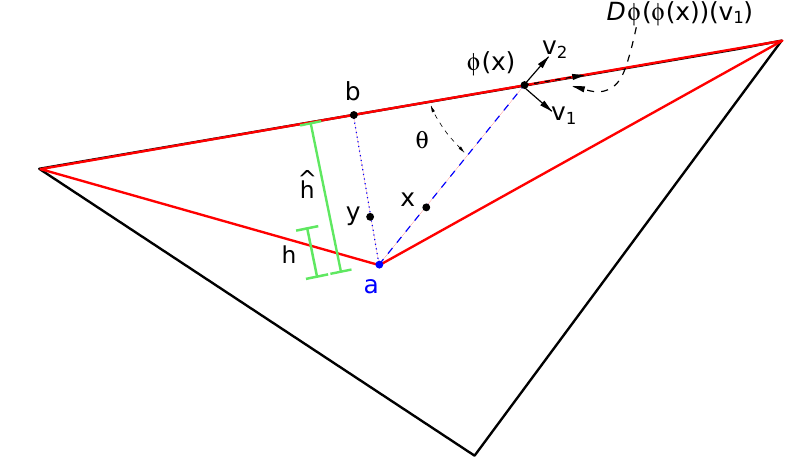}
  \caption{A $2$-dimensional illustration of the dilation
  calculation.}
  \label{fig:2d-diagram} 
  \end{figure}

  Choose an orthogonal basis $\{\vw_1,...,\vw_{\ell-2}\}$ for
  $\spn(\vv_1,\vv_2)^{\perp}$.  Note that $\spn(\vw_1, ...,$
  $\vw_{\ell-2}) \subset E_{\ell-1}$. Let $\vw'$ be a unit vector in
  $\spn(\vw_1,...,\vw_{\ell-2})^{\perp}\cap E_{\ell-1}$ parallel
  to $\,\phi(\vx) - \vb$. Then $\{\vv_1, \vw_1, ..., \vw_{\ell-2},$
  $\vv_2\}$ is an orthogonal basis for $\R^{\ell}$, and $\phi_{\pi}$
  is given by
  \begin{equation} \label{eq:orthbasisRl}
     \begin{array}{rcl}		
     \phi_{\pi}(\vv_1) & = &  \alpha \vw',\\
     \phi_{\pi}(\vw_i) & = & \vw_i, \; i \in \{1,...,\ell-2\},~~\mbox{ and }\\
     \phi_{\pi}(\vv_2) & = & 0, 
     \end{array}
  \end{equation}
  where $\alpha = \hat{r}/\hat{h}$. Notice that the above set up works
  everywhere except when $\phi(\vx) = \vb$, in which case we obtain an
  orthogonal projection for $\phi_{\pi}(\vx)$ along
  $\vb-\va$. Choosing coordinates for the tangent spaces of $\sigma$
  and $\tau_{\vx}$ to be $\{\vv_1, \vw_1, ..., \vw_{\ell-2}, \vv_2\}$
  and $\{\vw', \vw_1,..., $ $\vw_{\ell-2}\}$, respectively, we get
  from Equation~(\ref{eq:orthbasisRl}) that $D \phi_{\pi}(\vx,\va)$ is
  the $(\ell-1) \times \ell$ matrix given as
  \begin{equation} \label{eq:Dphimatrix}
  D\phi_{\pi}(\vx,\va) = 
   \begin{bmatrix} %\left( \begin{array}[htp!]{llllll} 
    \alpha & 0 & 0 & ... & 0 & 0 \\
	 0 & 1 & 0 & ... & 0 & 0 \\ 
    \vdots & &\ddots & & \vdots & \vdots \\ 
      0 & \hdots  & & & 1 & 0
    \end{bmatrix}. 
  \end{equation}
%    \end{array} \right). \]
%
\subsubsection{Bounding the Integral of the Jacobian}

We now present a series of bounds on integrals of the dilation of
$d$-volumes induced by the retraction. Since $\ell=d$ implies we are
already in the $d$-skeleton and no retraction is needed, we can assume
that $\ell > d$. We start with a bound on the maximum dilation of
$d$-volumes under the retraction $\phi$. $D\phi$ will denote the
\emph{tangent map} or \emph{Jacobian map} of $\phi$.
\begin{definition} 
  Let $J_d\phi(\vx,\va)$ be the maximum dilation of $d$-volumes
  induced by $D\phi(\vx,\va)$ at $\vx$.
\end{definition}

We will use the definitions and results on $D\phi_{\pi}(\vx,\va)$ in
$\ell$-dimension given above. In particular, recall that
$\diam(\sigma)$ is the diameter of $\sigma$, $\hat{h} = \|b-a\|$ and $h = \|y-a\|$.

  \begin{lemma} \label{lem:bnddjacob} For any center $\va$ and any
  point $\vx \neq \va$ in the $\ell$-simplex $\sigma$ with $d < \ell
  \leq p=d+k$,
  \[ J_d \phi(\vx,\va) \leq \left(\frac{\hat{h}}{h}\right)^d
  \frac{\diam(\sigma)}{\hat{h}}. \]
  \end{lemma}
  \begin{proof}
  Following Equation~(\ref{eq:phixa}), we seek bounds on $D
  \phi_{\delta}(\vx,\va)$ and $D \phi_{\pi}(\vx,\va)$. Since
  $D\phi_{\delta} (\vx,\va)$ simply scales by $\hat{r}/r = \hat{h}/h$,
  the expansion of $d$-volume of any $d$-hyperplane by
  $D\phi_{\delta}(\vx, \va)$ is by a factor of $(\hat{h}/h)^d$.  On
  the other hand, bounding the dilation that $D\phi_{\pi}(\vx,\va)$
  can cause in $d$-hyperplanes is a little more involved. We seek a
  bound on
  \begin{equation} \label{eq:dilratio}
  \frac{\sqrt{\det(\, (D\phi_{\pi}(\vx,\va) U)^T (D\phi_{\pi}(\vx,\va)
  U) \,)}} {\sqrt{\det(U^T U)}}
  \end{equation}
  for all $\ell \times d$ matrices $U$. Using the generalized
  Pythagorean theorem \cite[Section 1.5]{KrantzParks2008}, we get
  \[ \det(U^T U) =  \sum_{\lambda\in\Lambda} (\det(U_{\lambda}))^2 \]
  where submatrix $U_{\lambda}$ consists of the $d$ rows of $U$
  specified by the set of index maps $\Lambda$ given as
  \[ 
  \lambda \in \Lambda\equiv \{f | f:[1,...,d] \rightarrow
  [1,...,\ell], ~f \text{ is one to one and increasing}\}.
  \]
  A similar result holds for $\det( (D\phi_{\pi}(\vx,\va) U)^T
  (D\phi_{\pi}(\vx,\va) U) )$, with the functions $f$ considered
  mapping $[1,...,d]$ to $[1,...,\ell-1]$.

  Observe that multiplying by $D\phi_{\pi}(\vx,\va)$ (Equation
  \ref{eq:Dphimatrix}) just scales the first row of $U$ by $\alpha$
  and removes the last row. Thus $\alpha\det(U_{\lambda}) \geq
  \det((D\phi_{\pi}(\vx,\va)U)_{\lambda})$, which implies that
  $\alpha$ is a bound on the ratio in Equation~(\ref{eq:dilratio}).
  Thus we have that
  \[ J_d \phi(\vx,\va) = \left( \frac{\hat{r}}{r}\right)^d
  \frac{\norm{\phi(\vx,\va) - \va}} {\norm{\vb-\va}} = \left(
  \frac{\hat{h}}{h} \right)^d \frac{\norm{\phi(\vx,\va) -
  \va}}{\norm{\vb-\va}} \leq \left(\frac{\hat{h}}{h}\right)^d
  \frac{\diam(\sigma)}{\hat{h}} \]
  holds for all $\vx$ and $\va$ in $\sigma$, where $\diam(\sigma)$ is
  the diameter of the $\ell$-simplex $\sigma$. \end{proof}

  Next we describe a bound on the integral of $J_d \phi(\vx,\va)$ over
  the entire $\ell$-simplex, for a fixed center $\va$. We will find
  that this bound is independent of the position of $\va$. Recall that
  $\perimeter(\sigma)$ and $\per(\sigma)$ denote the perimeter of
  $\ell$-simplex $\sigma$ and the $(\ell-1)$-dimensional volume of the
  perimeter, respectively, and $\intrr(\sigma)$ its interior.

  \begin{lemma} \label{lem:bndJxfixa} 
    For any fixed center $\va$ in the $\ell$-simplex $\sigma$ with $d
    < \ell \leq p=d+k$,
    \[ \int_{\intrr(\sigma)} J_d \phi(\vx,\va) \, {\rm d}{\mathcal
    L}^{\ell}(\vx) \, \leq \, \diam(\sigma) \per(\sigma). \]
  \end{lemma}
  \begin{proof} 
  Consider the $(\ell-1)$-dimensional faces $\tau_j$ of $\sigma$, with
  $\perimeter(\sigma) = \{ \cup_j \tau_j \,|\, \tau_j \in \sigma,\,
  \dim(\tau_j)=\ell-1\}$. Let $\sigma_j$ denote the $\ell$-simplex
  generated by $\va$ and $\tau_j$. Then
  \[\int_{\intrr(\sigma)} J_d \phi(\vx,\va) \, {\rm d}{\mathcal
  L}^{\ell}(\vx) = \sum_j \int_{\intrr(\sigma_j)} J_d \phi(\vx,\va) \,
  {\rm d}{\mathcal L}^{\ell}(\vx).\]
  Let $\tau_j(h)$ denote the $(\ell-1)$-simplex formed by the
  intersection of $\sigma_j$ and the $(\ell-1)$-hyperplane parallel to
  $\tau_j$ at a distance $h$ from $\va$. Thus, $\tau_j(\hat{h})$ is
  $\tau$ itself. We observe that our bound on $J_d \phi(\vx,\va)$ is
  constant in $\tau_j(h)$ for any $h$. The $(\ell-1)$-dimensional
  volume of $\tau_j(h)$ is given by
  \[ \vol_{\ell-1}(\tau_j(h)) =
  \left(\frac{h}{\hat{h}}\right)^{\ell-1} \vol_{\ell-1} (\tau_j).\]
  Using the bound on $J_d \phi(\vx,\va)$ from Lemma
  \ref{lem:bnddjacob}, and noting that $\diam(\sigma_j) \leq
  \diam(\sigma)~ \forall \,j$, we get
  \[ \int_{\intrr(\sigma_j)} J_d \phi(\vx,\va) \, {\rm d}{\mathcal
  L}^{\ell}(\vx) \leq \int_0^{\hat{h}}
  \left(\frac{h}{\hat{h}}\right)^{\ell-1} \vol_{\ell-1} (\tau_j) \,
  \left(\frac{\hat{h}}{h}\right)^d \frac{\diam(\sigma)}{\hat{h}} \,
  {\rm d}h = \frac{\vol_{\ell-1}(\tau_j) \diam(\sigma)}{\ell-d}.\]
  Summing this quantity over all $\tau_j \in \perimeter(\sigma)$ and
  replacing $\ell-d \geq 1$ with $1$ gives the overall
  bound.  \end{proof}

  We now bound the integral of $J_d \phi(\vx,\va)$ over centers $\va$
  with a fixed $\vx$ that we are retracting onto $\perimeter(\sigma)$.
  Examination of the corresponding proof for the original deformation
  theorem \cite[Section 7.7]{KrantzParks2008} shows that symmetry of
  the cubical mesh plays a very special role, which cannot be
  duplicated in the case of simplicial complex. In particular, we must
  avoid integrating over $\va$ close to the perimeter of
  $\sigma$. Hence we integrate over as big a region as we can while
  still avoiding a neighborhood of the perimeter. As in the statement
  of the main Theorem \ref{thm:simpldeform}, let $\hat{\cal B}_\sigma$
  be the largest ball inscribed in $\sigma$, ${\cal B}_\sigma$ be the
  ball with half the radius and same center as $\hat{\cal B}_\sigma$,
  and $r_{\sigma}$ be the radius of ${\cal B}_\sigma$.
  \begin{lemma} \label{lem:bndJafixx}
    For any point $\vx$ in the $\ell$-simplex $\sigma$ with $d < \ell
    \leq p=d+k$,
     \[ \int_{{\cal B}_\sigma} J_d \phi(\vx,\va) \, {\rm d}{\cal
     L}^{\ell} (\va) \, \leq \, \diam(\sigma) \per(\sigma) +
     \vol_{\ell}({\cal B}_\sigma) \frac{\diam(\sigma)}{r_{\sigma}}. \]
  \end{lemma}
  \begin{proof} 
  Similar to the subsimplices of $\sigma$ considered in the Proof of
  Lemma \ref{lem:bndJxfixa}, let $\sigma_j$ now denote the
  $\ell$-simplex formed by $\vx$ and $\tau_j \in \perimeter(\sigma)$.
  In order to derive an upper bound, we integrate instead over regions
  that are by construction bigger than these subsimplices of $\sigma$.
  Denoting the simplex $\sigma_j$ as Region 1, we define Regions 2 and
  3 as follows. We refer the reader to Figure~\ref{fig:integrate-a}
  for an illustration of this construction. Let $\sigma'_j$ be the
  reflection of $\sigma_j$ through $\vx$, and similarly, let $\tau'_j$
  be the reflection through $\vx$ of $\tau_j$. We define $\sigma'_j$
  as Region 2. Notice that unlike Region 1, Region 2 need not be
  contained fully in $\sigma$. As defined in Section
  \ref{sssec:retrsetup}, let $\vz$ be the unit vector normal to the
  $(\ell-1)$-hyperplane containing $\tau_j$ pointing into $\sigma$. We
  define Region 3 as the $\ell$-dimensional set $\,\tau'_j +
  [0,\diam(\sigma)]\vz$, as illustrated in
  Figure~\ref{fig:integrate-a}.

  Note that the union of all Region 2's and Region 3's cover $\sigma$.
  By an argument almost identical to that above, we have the following
  upper bound on the integrand in question.
  \begin{equation*} \label{eq:int-a-bound} 
  \left(\frac{h'}{h}\right)^d \frac{\norm{\phi(\vx,\va) -
  \va}}{\hat{h}} \, \leq \, \left( \frac{h'}{h} \right)^d
  \frac{\diam(\sigma)}{\hat{h}}.
  \end{equation*}
  \begin{figure}[ht!]
  \centering 
  \includegraphics[scale=0.6, trim=1in 3.2in 1.5in 1in,
  clip]{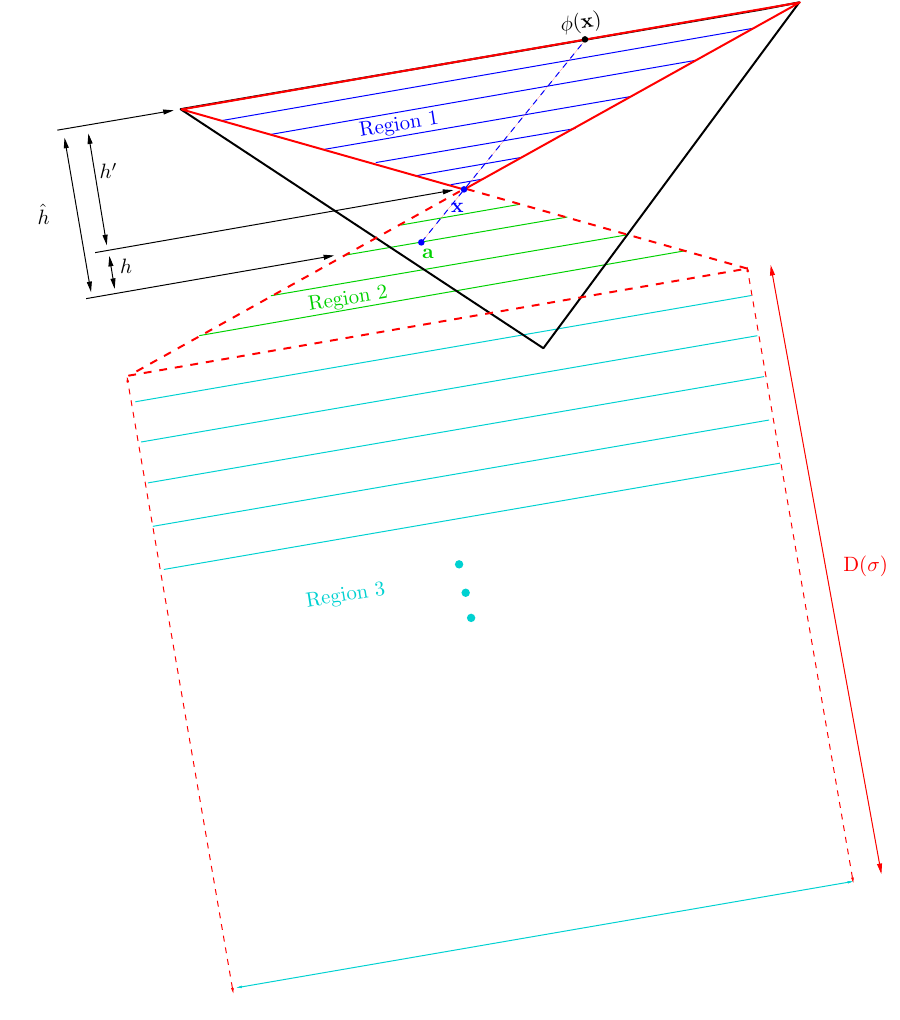} 
  \caption{Illustration for integration of Jacobian bound over centers
  $\va$ instead of $\vx$'s. For the case of the triangle shown, there
  will be 3 sets of 3 regions. In general there will be 3 regions for
  every face of the simplex.}
  \label{fig:integrate-a} 
  \end{figure}
  Integrating the second of these two terms over Region 2 and summing
  the integral over all such Regions 2 for all faces $\tau_j$, we get
  the upper bound of $\diam(\sigma)\per(\sigma)$. Here we use the same
  arguments as the ones employed in Lemma \ref{lem:bndJxfixa}. Region
  2 alone is not guaranteed to cover ${\cal B}_{\sigma}$ as some of
  ${\cal B_\sigma}$ may occupy parts of Region 3. Since $\va \in {\cal
    B}_\sigma$, we have $\hat{h} > r_{\sigma}$, and $h' \leq h$ when
  $\va \in $Region 3, so that
  \[\left(\frac{h'}{h}\right)^d \frac{\diam(\sigma)}{\hat{h}} \leq
  \frac{\diam(\sigma)}{r_{\sigma}}.\]
  Combining the above estimates while integrating over all such
  Regions 2 and 3 gives us the bound specified in the
  Lemma. \end{proof}

\subsubsection{Bounding the pushforwards of the current}

  We consider the $d$-current $T$, and employ the bounds on the
  Jacobian of retraction described above to the pushforwards of $T$
  and its boundary $\boundary T$ on to the simplicial complex $K$.
  Our treatment of the pushforwards essentially follows the
  corresponding results of Krantz and Parks for the case of square
  grid \cite[Pages 218--219]{KrantzParks2008}. We denote by $\norm{T}$
  the total variation measure of the current $T$, which is determined
  by the identity
  \[
  \norm{T}(W) = \sup_{\text{\parbox{9em}{\centering$\omega \in
        \mathcal{D}^d$, $\norm{\omega} = 1$,\\$\spt\omega \subset
        W$}}} T(\omega).
  \]
  \begin{lemma} \label{lem:retTbnd} 
  Suppose $K$ is a $p$-dimensional simplicial complex with $p = d+k$
  for $k \geq 1$. Consider the stepwise retraction of the $d$-current
  $T\subset K$ (the $(d-1)$-current $\boundary T \subset K$) onto the
  $d$-skeleton of $K$ (respectively, the $(d-1)$-skeleton of
  $K$). Each step of the retraction on to the perimeter of an
  $\ell$-simplex $\sigma$ for $d < \ell \leq p$ (respectively, $d \leq
  \ell \leq p $) increases the mass of $T$ or $\boundary T$ by at most
  a factor of
  \[ 4 \upvartheta_K = 4(\upkappa_1 + \upkappa_2) = 4 \max_{\sigma \in
  K} \left( \frac{\diam(\sigma)\per(\sigma)}{\vol_{\ell}({\cal
  B}_\sigma)} + \frac{\diam(\sigma)}{r_{\sigma}} \right).\]
  \end{lemma}

  \begin{proof} 
  Using Fubini's theorem \cite[Page 26]{KrantzParks2008} and applying
  the bound in Lemma \ref{lem:bndJafixx}, we get
  \[ \int_{{\cal B}_\sigma} \int_{\sigma} J_d \phi(\vx,\va) \, {\rm
  d}\norm{T}(\vx) \, {\rm d} {\cal L}^{\ell}(\va) = \int_{\sigma}
  \int_{{\cal B}_\sigma} J_d \phi(\vx,\va) \, {\rm d}{\cal
  L}^{\ell}(\va) \, {\rm d} \norm{T}(\vx) \, \leq \,
  \upvartheta_{\sigma} \mass(T|_{\sigma}), \]
  where $\upvartheta_{\sigma} = \diam(\sigma)\per(\sigma) +
  \vol_d({\cal B}_\sigma) (\diam(\sigma)/r_{\sigma})\,$ and
  $T|_{\sigma}\,$ is the portion of the current $T$ restricted to the
  simplex $\sigma$.  Consider the subset of ${\cal B}_{\sigma}$
  defined as
  \[ H_T = \left\{ \va \in {\cal B}_{\sigma} \, \big| \, \int_{\sigma}
  J_d \phi(\vx,\va) \, {\rm d}\norm{T}(\vx) > \frac{
  4\upvartheta_{\sigma} \mass(T|_{\sigma})}{\vol_{\ell}({\cal
  B}_{\sigma})} \right\}. \]
  Then $\vol_{\ell}(H_T) \leq (1/3) \vol_{\ell}({\cal
  B}_{\sigma})$. Similarly we define $H_{\boundary T}$ for the
  pushforward of $\boundary T$ and get $\vol_{\ell}(H_{\boundary T})
  \leq (1/3) \vol_{\ell}({\cal B}_{\sigma})$. Then the set $\,{\cal
  B}_{\sigma} \setminus \{H_T \cup H_{\boundary T} \}\,$ defines a
  subset of ${\cal B}_{\sigma}$ with positive measure, with the
  centers $\va$ in this subset satisfying $\, \int_{\sigma} J_d
  \phi(\vx,\va) \, {\rm d} \norm{T}(\vx) \, \leq \, 4
  \upvartheta_{\sigma} \mass(T|_{\sigma})/\vol_{\ell}({\cal
  B}_{\sigma})\,$ and $\, \int_{\sigma} J_d \phi(\vx,\va) \, {\rm d}
  \norm{\boundary T}(\vx) \, \leq \, 4 \upvartheta_{\sigma}
  \mass(T|_{\sigma})/\vol_{\ell}({\cal B}_{\sigma})$. Hence we can
  choose centers to retract from in each simplex $\sigma$ such that
  the expansion of mass of the current restricted to that simplex is
  bounded by $4 \upvartheta_{\sigma} / \vol_{\ell}({\cal
  B}_{\sigma})$. The bound specified in the Lemma follows when we
  consider retracting the entire current over multiple simplices in
  $K$, and set $\upvartheta_K = \max_{\sigma \in K}
  \upvartheta_{\sigma}$ as the generic upper bound that holds for all
  simplices in $K$.
  \end{proof}

  \paragraph{Bound on complete sequence of retractions.} We can apply
  the bound specified in Lemma \ref{lem:retTbnd} over multiple levels
  $\ell$. Pushing $T$ onto the $d$-skeleton of $p$-complex $K$
  multiplies the mass of $T$ by a factor of at most
  $(4\upvartheta_K)^k$. Likewise, pushing $\boundary T$ on to the
  $(d-1)$-skeleton multiplies the mass of $\boundary T$ by a factor of
  at most $(4\upvartheta_K)^{k+1}$.

\subsubsection{Bounding the distance between the current and its simplicial approximation}

  In the final step, we construct the simplicial current $P$
  approximating the original current $T$, and bound the flat norm
  distance between the two. Since we are now considering retraction
  maps over many simplices simultaneously, we let $\phi_i$ denote the
  global projection from the $(p-i+1)-$skeleton to the
  $(p-i)-$skeleton, suppressing the particular $\vx$ and $\va$.  We
  denote the composition of all these steps as $\psi^1 \equiv \phi_k
  \circ \dots \circ \phi_1$ and hence we map $T$ forward by $\psi^1$,
  picking centers (see Lemma \ref{lem:bndJafixx}) to project from in
  each step and in each simplex.  We pick each of these centers such
  that the retractions map $\boundary T$ with bounded amplification of
  mass as well (see Lemma \ref{lem:retTbnd}).

  The homotopy formula \cite[Section 7.4.3]{KrantzParks2008} states
  that given a smooth homotopy $g$ from $f_0$ to $f_1$ where $f_0, f_1
  : U \subseteq \R^{n_0} \to \R^{n_1}$ are smooth functions with
  $g(0,x) = f_0(x)$ and $g(1,x) = f_1(x)$, if $T$ is a $d$-current and
  $f^{-1}(F) \cap \spt f$ is compact for every compact set $F
  \subseteq \R^{n_1}$, we have that the difference in pushforwards of
  $T$ under $f_1$ and $f_0$ is given by
  \[
    {f_1}_{\#}(T) - {f_0}_{\#}(T) = \boundary g_{\#}([0,1] \times T) +
    g_{\#}([0,1] \times \boundary T).
    \]
  Define the homotopy $g(\gamma,\vx) = \gamma\vx + (1-\gamma)
  \psi^1(\vx)$ for $\gamma \in [0,1]$. Then the homotopy formula gives
  \[ T - \psi^1_{\#}(T) = \boundary g_{\#}([0,1] \times T) +
  g_{\#}([0,1] \times \boundary T).\]
  We define $R = g_{\#}([0,1] \times T)$ and $Q_1 = g_{\#}( [0,1]
  \times \boundary T)$.  Then we get
  \begin{equation} \label{eq:Tminuspsi}
    T - \psi^1_{\#}(T) = \boundary R + Q_1. 
  \end{equation}

  \noindent Finally, we map $\psi^1_{\#}(\boundary T)$ forward to the
  $(d-1)$-skeleton of simplicial complex $K$ with $\phi = \phi_{k+1}$
  to get $\psi^2_{\#}(\boundary T) = \phi_{\#}(\psi^1_{\#}( \boundary
  T))$. For this purpose, consider the homotopy $h(\gamma,\vx)$ from
  $\psi^1_{\#}(\boundary T)$ to $\psi^2_{\#}(\boundary T)$, i.e.,
  \[ h(\gamma,x) = \gamma\psi^1_{\#}(\vx) + (1-\gamma)
  \psi^2_{\#}(\vx) ~ \mbox{ for }~\gamma \in [0,1]. \]
  We define
  \begin{equation} \label{eq:defP}
    P = \psi^1_{\#}(T)- h_{\#}([0,1]\times\psi^1_{\#}(\boundary T)).
  \end{equation}
  $P$ is a $d$-current whose boundary $\boundary P$ is contained in the
  $(d-1)$-skeleton of $K$. Define $Q_2 = h_{\#}([0,1] \times
  \psi^1_{\#}(\boundary T))$.  Using the homotopy formula, we get
   \begin{align*} % \label{eq:homotopy}
    \boundary P & = \boundary\left( \, \psi^1_{\#}(T)-
    h_{\#}([0,1]\times\psi^1_{\#}(\boundary T)) \, \right) \\
              & =  \, \psi^1_{\#}(\boundary T)-
              \boundary h_{\#}([0,1]\times\psi^1_{\#}(\boundary T)) \, \\
              & = \psi^2_{\#}(\boundary T) \subset (d-1)\text{-skeleton of } K.
   \end{align*}
  Equation (\ref{eq:defP}) gives $\psi^1_{\#}(T) = P + Q2$.  Defining
  $Q = Q_1 + Q_2$, Equation (\ref{eq:Tminuspsi}) gives
  \begin{align*}
    T - (P+Q2) & = \boundary R + Q_1, ~~\mbox{hence } \\
    T - P      & = \boundary R + Q.  
  \end{align*}

  Finally, we apply the bounds on the retraction described in Lemma
  \ref{lem:retTbnd} and the paragraph following this Lemma to the
  masses of the pushforwards. Noticing that $\diam(\sigma) \leq
  \Updelta$ for all $\sigma \in K$, we get the following bounds, which
  finish the proof of our simplicial deformation theorem (Theorem
  \ref{thm:simpldeform}). 
  \begin{align*}
  \mass(P) 	    & \leq  (4\upvartheta_{K})^k \mass(T) + 
			\Updelta(4\upvartheta_{K})^{k+1}\mass(\boundary T)\\
                    &  =    (4\upvartheta_{K})^k\left(\mass(T) + 
			\Updelta(4\upvartheta_{K})\mass(\boundary T)\right),\\
  \mass(\boundary P) & \leq  (4\upvartheta_{K})^{k+1}\mass(\boundary T),\\
  \mass(R)          & \leq   \Updelta \mass(\psi^1_{\#}(T)) \\
               	    & \leq   \Updelta (4\upvartheta_{K})^{k} \mass(T),~\mbox{ and } \\
  \mass(Q)          & \leq  \Updelta(4\upvartheta_{K})^{k} \mass(\boundary T)
			+ \Updelta(4\upvartheta_{K})^{k+1} \mass(\boundary T)\\
                    &  =    \Updelta(4\upvartheta_{K})^{k}(1+4\upvartheta_{K})
			\mass(\boundary T).
  \end{align*}
\qed

%  \bigskip

  \begin{remark} 
  \label{rem:subdivisionregularity}
    The influence of {\em Simplicial regularity} as measured by
    $\upkappa_1$ and $\upkappa_2$ is clearly revealed by the statement
    of our deformation theorem (Theorem \ref{thm:simpldeform}).
    Explicit constants are a simple yet useful part of the result; as
    observed above in Remark~\ref{rem:defthm1}, the statement of this
    theorem leads to an easy observation that the flat norm distance
    between $T$ and $P$ can be made a small as desired by subdividing
    the simplicial complex in a manner that keeps the regularity
    constants bounded.  This can be done, for example, by using the
    subdivision algorithm of Edelsbrunner and Grayson \cite{EdGr2000}.
  \end{remark}

  \begin{remark}
    We did not explicitly discuss the case of $0$-dimensional
    currents.  In this case, the bounds on mass expansion are all
    equal to one.
  \end{remark}

\subsection{Comparison of Bounds of Approximation} \label{ssec:compbnds}

Sullivan studied the deformation of integral currents on to the
skeleton of a cell complex, which is composed of compact convex
sets. He presented a deformation theorem for deforming integral
currents on to the boundary of a cell complex \cite[Theorem
  4.5]{Sullivan1990}. For ease of comparison, we use {\em our}
notation to restate the bounds given by Sullivan for deforming a
$d$-current $T$ to a polyhedral current $P$ in the boundary of a cell
complex in $\R^q$. Recall that in our simplicial deformation theorem
(Theorem \ref{thm:simpldeform}), the simplicial complex considered has
dimension $p$ and is embedded in $\R^q$ for $q \geq p$.  Furthermore,
$\upkappa_1$, $\upkappa_2$, $\Updelta$, and $\upvartheta_K$ are
simplicial regularity constants.  We also note that even though
Sullivan stated his results for full-dimensional complexes and the
standard flat norm, it is straightforward to extend them to lower
dimensional complexes and the flat norm with scale:
\begin{align}
  \mass(P) & ~\leq~ {p \choose d} \left( 2d \left(\frac{d+1}{2d}
  \upkappa_2 \right)^{d+1} \right)^{p-d+1}
  \mass(T) \label{eq:SullbndMP}, \\
  \vspace*{-0.1in} \nonumber \\
  \mass(\boundary P) & ~\leq~ {p \choose d-1} \left(2d \left(
  \frac{d+1}{2d} \upkappa_2 \right)^{d} \right)^{p-d+1}
  \mass(\boundary T), ~~~ \mbox{and} \label{eq:SullbndMbdyP} \\
  \vspace*{-0.5in} \nonumber \\
 \F_\lambda(T,P) & ~=~\lambda^d \cdot \F_1(T/\lambda,
 P/\lambda)\nonumber\\ &~\leq~ \lambda^d \cdot (p-d+1) \Updelta \, (
 \, \mass(P/\lambda) + \mass(\boundary P/\lambda) \,)\nonumber\\ &~=
 \lambda^d \cdot (p-d+1) \Updelta \, ( \,\lambda^{-d} \cdot \mass(P) +
 \lambda^{1-d} \cdot \mass(\boundary P) \,)\nonumber\\ &~= (p-d+1)
 \Updelta \, ( \, \mass(P) + \lambda \mass(\boundary P) \,).
\label{eq:SullbndFPbdyP}
\end{align}
Our results corresponding to the first two bounds in Equations
(\ref{eq:SullbndMP}) and (\ref{eq:SullbndMbdyP}) are presented in
Equations (\ref{eq:simpdefthmMP}) and (\ref{eq:simpdefthmMbdyP}) in
Theorem \ref{thm:simpldeform}, which we repeat here with the
substitution $k = p-d$.
  \begin{align}
    \mass(P) & \leq (4\upvartheta_K)^{p-d} \mass(T)+ \Updelta(4\upvartheta_{K})^{p-d+1}\mass(\boundary T),
    \tag{\ref{eq:simpdefthmMP} revisited}\\
    \mass(\boundary P) & \leq (4\upvartheta_K)^{p-d+1} \mass(\boundary T), 
    \tag{\ref{eq:simpdefthmMbdyP} revisited}\\
    \mass(R) & \leq \Updelta(4\upvartheta_K)^{p-d} \mass(T), \mbox{ and}
    \tag{\ref{eq:simpdefthmMR} revisited} \\
    \mass(Q) & \leq  \Updelta(4\upvartheta_{K})^{p-d}(1+4\upvartheta_{K})\mass(\boundary T). 
    \tag{\ref{eq:simpdefthmMQ} revisited}
  \end{align}
To obtain the flat
norm distance corresponding to the third bound given by Sullivan in
Equation (\ref{eq:SullbndFPbdyP}), we use the definition of flat norm
distance between two currents specified in Equation
\eqref{eq:flatdist}. Using $T-P = \boundary Q + R$, we combine two of
our bounds specified in Equations \eqref{eq:simpdefthmMR} and
\eqref{eq:simpdefthmMQ} to get
\begin{align*}
\F_\lambda(T,P) & ~\leq~ \Updelta (4\upvartheta_{K})^{p-d} \, \left( \, \mass(T) + \lambda(1 
	+ 4 \upvartheta_{K}) \mass(\boundary T) \, \right). \hspace*{0.5in}
\end{align*}

To gain a better understanding of how the two sets of bounds compare,
we compute these bounds explicitly for the case of a $2$-current in a
regular tetrahedral complex (thus, $p = 3$ and $d = 2$). Notice that
this instance is close to a best case for Sullivan's bounds, as less
regular complexes affect them more severely. With this point in mind,
we present in Table \ref{table:boundcomparison} our bounds and
Sullivan's bounds on both a regular tetrahedral complex and one on
which we stretch the regular tetrahedra by a factor of 10 in a
direction normal to one of their faces (i.e., turn them into skinny,
spike-like simplices). 
\begin{table}[h]
\centering
\renewcommand{\arraystretch}{2}
\begin{tabular}{@{}rccr@{}}\toprule
\\[-7ex]
%& \multicolumn{2}{|c|}{Regular tetrahedra} & \multicolumn{2}{|c|}{Stretched tetrahedra}\\
Quantity & Sullivan's bound & Our bound\\
\hline

\multicolumn{3}{c}{Regular tetrahedra}\\[-1ex]

$M(P)$ & $(1.2 \times 10^5)\, \mass(T)$ & \twolinecelleq{(1.6 \times 10^3) \, \mass(T)}{ + (2.5 \times 10^6)\, \Updelta\mass(\boundary T)}
\\

$M(\boundary P)$ & $(8.7 \times 10^3)\,\mass(\boundary T)$ & $(2.5 \times 10^6)\,\mass(\boundary T)$ \\

$\F_\lambda(T, P)$ & \twolinecelleq{(2.4 \times 10^5)\,\Updelta\mass(T)}{+(1.7 \times
10^3)\,\Updelta\lambda\mass(\boundary T)} & \twolinecelleq{(1.6 \times
10^3)\,\Updelta\mass(T)}{ + (2.5 \times 10^6) \,
\Updelta\lambda\mass(\boundary T)} \vspace*{0.1in} \\[-2ex]

\multicolumn{3}{c}{Stretched tetrahedra}\\[-1ex]

$M(P)$ & $(5.5 \times 10^9)\,\mass(T)$ & \twolinecelleq{(3.7 \times 10^4)\,\mass(T) }{+ (1.4 \times 10^9)\, \Updelta\mass(\boundary T)} \\

$M(\boundary P)$ & $(1.1 \times 10^7)\,\mass(\boundary T)$ & $(1.4 \times 10^9)\,\mass(\boundary T)$ \\

$\F_\lambda(T, P)$ & 
\twolinecelleq{(1.1 \times 10^{10})\Updelta\mass(T)}{ + (2.3 \times 10^7)\Updelta\lambda\mass(\boundary T)}
%\parbox{1.8in}
%{
%\begin{align*}
%&(1.1 \times 10^{10})\Updelta\mass(T)\\
%+\,&(2.3 \times 10^7)\Updelta\lambda\mass(\boundary T)
%\end{align*}
%}
%$(1.09778 \times 10^{10})\Updelta\mass(T) + (2.26170 \times 10^7)\Updelta\lambda\mass(\boundary T)$
& \twolinecelleq{(3.7 \times 10^4)\,\Updelta\mass(T)}{ + (1.4 \times 10^9)\Updelta\lambda\mass(\boundary T)}\\
\bottomrule
\end{tabular}
\caption{Comparison of our bounds with those obtained by Sullivan for
  a 2-current in a (1) 3-complex of congruent regular tetrahedra and
  (2) a 3-complex of congruent stretched tetrahedra which are created
  by taking regular tetrahedra and multiplying their height by a
  factor of 10.}
\label{table:boundcomparison}
\end{table}

For the regular tetrahedral complex and the $\mass(P)$ bound, our
coefficient of $\mass(T)$ is more than $74$ times better, but we do
have a second term that can be quite large, but diminishes in
importance if the complex is subdivided appropriately (see Remark
\ref{rem:subdivisionregularity}).  In the stretched complex, our
coefficient on $\mass(T)$ is $1.5 \times 10^5$ times better,
indicating that our bound is better behaved for irregular complexes.
Our bound on $\mass(\boundary P)$ is about $290$ times worse than
Sullivan's for the regular tetrahedra, and about $120$ times worse for
the stretched complex.  For the flat norm bound in the regular
complex, we are about $148$ times better on the $\mass(T)$ term and
about $145$ times worse on the $\mass(\boundary T)$ term.  On the
stretched complex, our $\mass(T)$ coefficient is about $3 \times 10^5$
times better, and our $\mass(\boundary T)$ coefficient is about $60$
times worse.  We also note that in the case of the flat norm with
scale, our larger $\mass(\boundary T)$ coefficient becomes less
important for small $\lambda$.

\begin{remark}
  For the important case where $\boundary T$ is empty, i.e., when $T$
  is a cycle, we have $\mass(\boundary T) = 0$, and hence our bounds
  are uniformly better than Sullivan's.
\end{remark}
As compared to Sullivan, we are able to take advantage of our
simplicial setting to get better bounds on the mass expansion of $T$.
While our mass expansion bounds involving $\boundary T$ are currently
inferior to Sullivan's, we suspect our arguments can be tightened and
modified to obtain bounds that are better in all cases.  More
importantly, our bounds are less sensitive to simplicial irregularity.
Given the challenges inherent in creating meshes without slivers even
in three dimensions \cite{ChDeSh2012}, bounds that behave well in
their presence are highly desirable.

\section{Computational Results} \label{sec:compresults}

We illustrate computations of the \MSFN{} by describing the flat norm
decompositions of a $2$-manifold with boundary embedded in $\R^3$ (see
Figure \ref{fig:pyrsurf}). The input set has the underlying shape of a
pyramid, to which several peaks and troughs of varying scale, as well
as random noise, have been added. We model this set as a piecewise
linear $2$-manifold with boundary, and find a triangulation of the
same as a subcomplex of a tetrahedralization of the $2 \times 2 \times
2$ cube centered at the origin, within which the set is located. We
use the method of constrained Delaunay tetrahedralization
\cite{Si2010} implemented in the package TetGen \cite{tetgen} for this
purpose. We then compute the \MSFN{} decomposition of the input set at
various scale ($\lambda$) values. At high values, e.g., when
$\lambda=6$, the optimal decomposition resembles the input set with
the small kinks due to random noise smoothed out. At the other end,
for $\lambda=0.01$, the optimal decomposition resembles a flat
``sheet''. For intermediate values of $\lambda$, the optimal
decomposition captures features of the input set at varying scales.

\begin{figure}[htp!]
\centering 
\vspace*{0.5in}
 \includegraphics[scale=0.51, trim=1.2in 4.1in 1.5in 1.5in, clip]{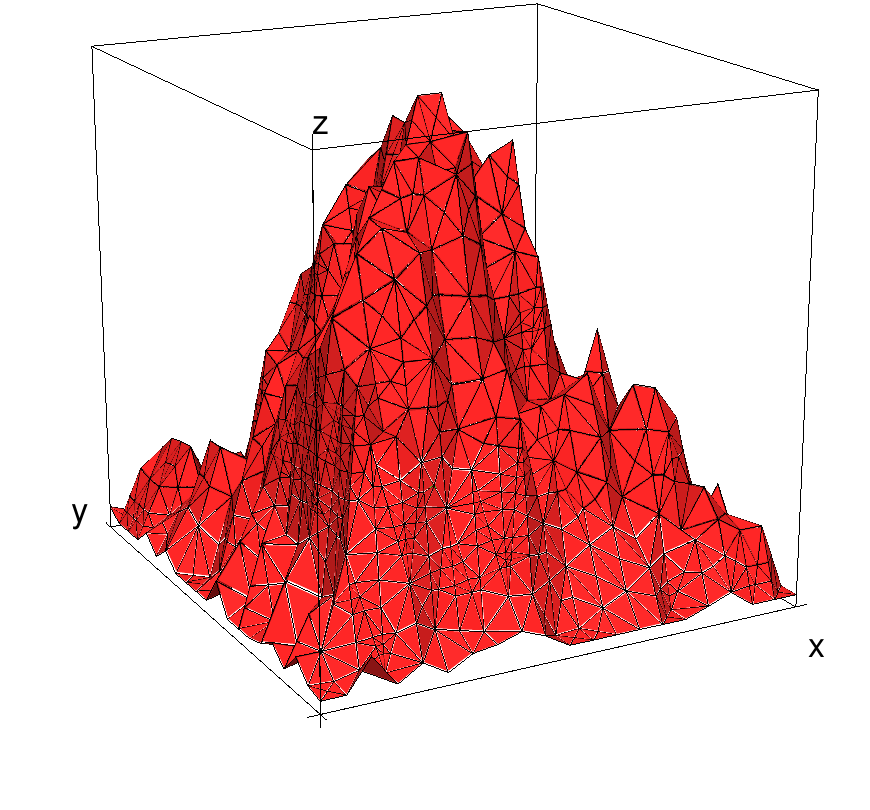} 
%\includegraphics[scale=0.6, trim=1.1in 1.1in 1in 1.5in, clip]{View1_Pyramid.eps} 
%\hspace*{0.05in}
\hfill
 \includegraphics[scale=0.51, trim=1.2in 4.1in 1.5in 1.5in, clip]{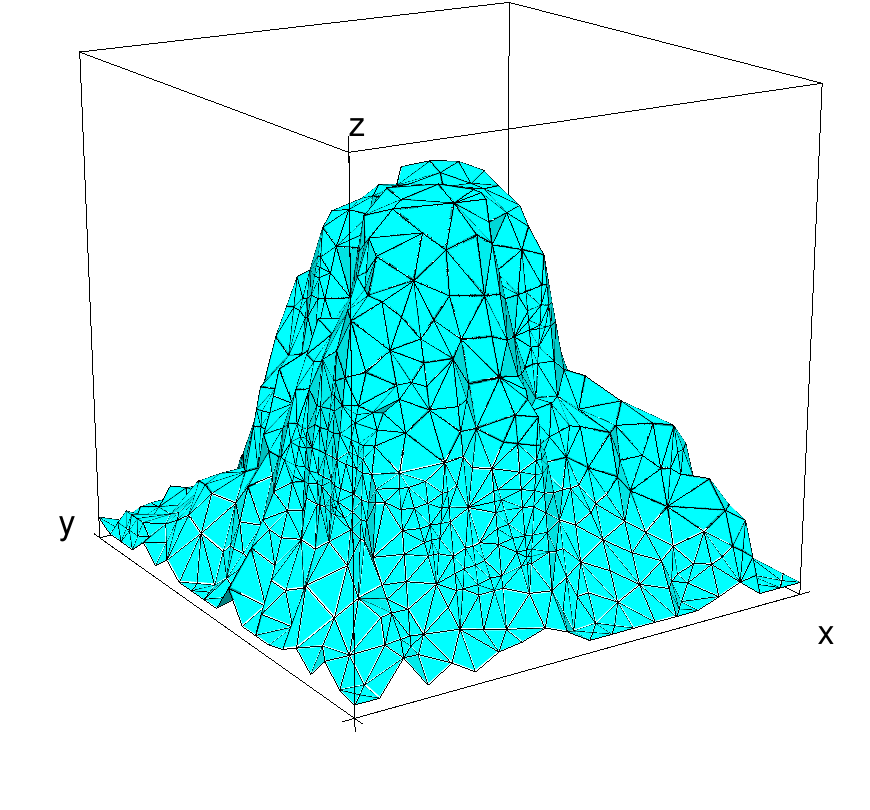} \\
 \includegraphics[scale=0.51, trim=1.2in 4.1in 1.5in 1.2in, clip]{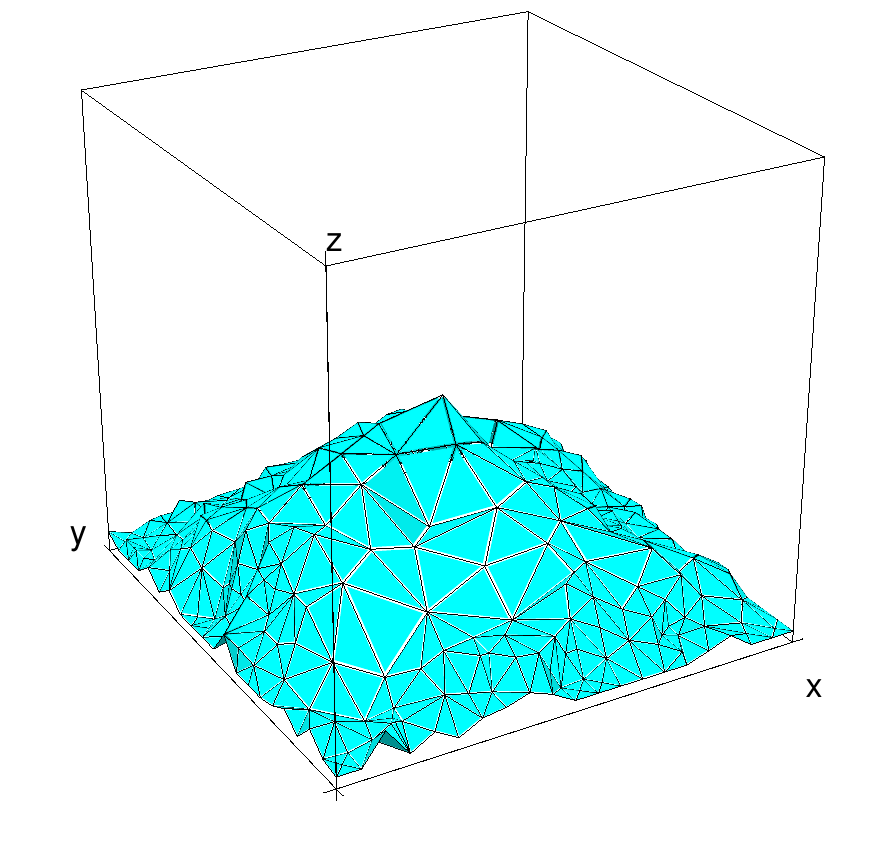} 
%\includegraphics[scale=0.55, trim=1.1in 1in 1in 1.4in, clip]{MSFNSurf_l2.eps} 
%\hspace*{0.05in}
\hfill
 \includegraphics[scale=0.51, trim=1.2in 4.1in 1.5in 1.2in, clip]{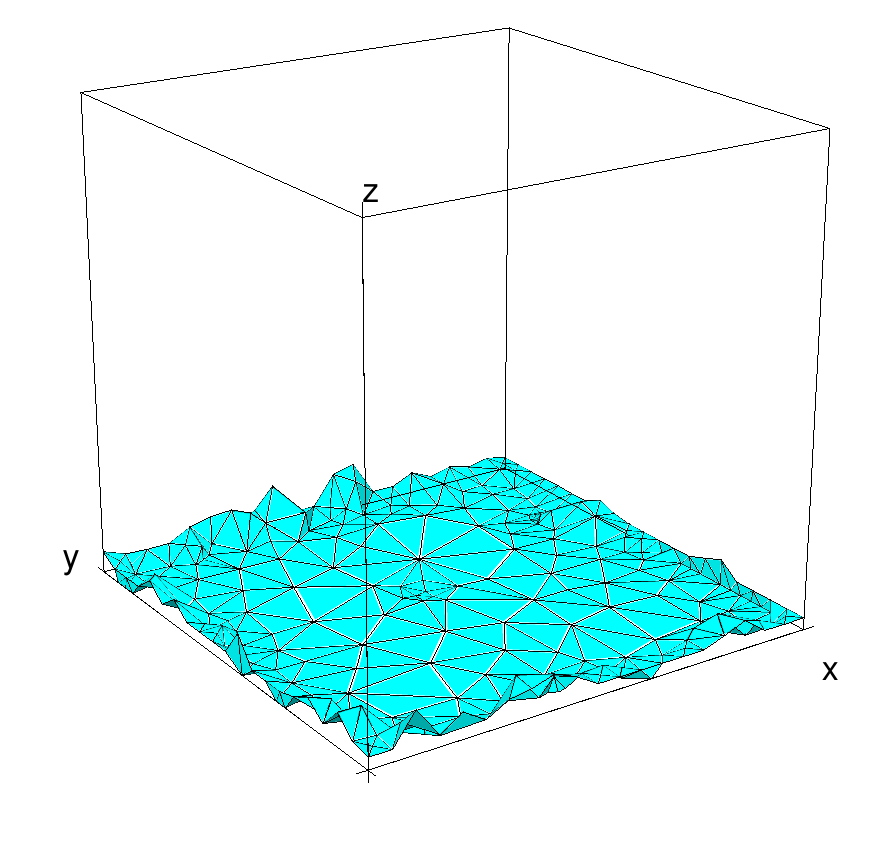} 
\vspace*{0.1in}
\caption{\label{fig:pyrsurf} Top left: A view of original 
pyramidal surface in three dimensions. The remaining three figures
show the flat norm decomposition for scales $\lambda=6$ (top right),
$\lambda=2$ (bottom left), and $\lambda=0.01$ (bottom right). See text
for further explanation. The images were generated using the package
TetView \cite{tetview}.}
\end{figure}

The entire $3$-complex mesh modeling the cube in question consisted of
14,002 tetrahedra and 28,844 triangles. For each $\lambda$,
computation of the \MSFN{} described above took only a few minutes on a
regular PC using standard functions from MATLAB. This example
demonstrates the feasibility of efficiently computing flat norm
decompositions of large datasets in high dimensions, for the purposes
of denoising or to recover scale information of the data.

\section{Discussion}

Our result on simplicial deformation (Theorem \ref{thm:simpldeform})
places the definition of the \MSFN{} into clear context. If a current
lives in the underlying space of a simplicial complex, we can deform
it to be a simplicial current on the simplicial complex, and do so
with {\em controlled} error. In fact, by subdividing the simplicial
complex carefully, we can move this error as close to zero as we
like. Since the \MSFN{} could be computed efficiently when the
simplicial complex does not have relative torsion, one could naturally
use our approach to compute the flat norm of a large majority of
currents in arbitrarily large dimensions. An important open question
in this context is whether the \MSFN{} of a current on a simplicial
complex {\em with} relative torsion could be approximated efficiently
by {\em coarsening} the complex so that the relative torsion is
removed. For instance, it has been observed recently that edge
contractions could remove existing relative torsion while preserving
the homology groups of the simplicial complex in certain cases
\cite{DeHiKrSm2013}.

The \MSFN{} problem, similar to the recent results on the \OBCP{}
\cite{DuHi2011}, apply notions from algebraic topology and discrete
optimization to problems from geometric measure theory such as flat
norm of currents and area-minimizing hypersurfaces. What other classes
of problems from the broader area of geometric analysis could we
tackle using similar approaches? One such question appears to be the
following: under what conditions is the flat norm decomposition of an
integral current guaranteed to be another integral current?  Working
in the setting of simplicial complexes, results on the existence of
integral optimal solutions for instances of ILPs with integer
right-hand side vectors may prove useful in answering this question.

While $L^1$TV and flat norm computations have been used widely on data
in two dimensions, such as images, the \MSFN{} opens up the
possibility of utilizing flat norm computations for higher dimensional
data. Similar to the flat norm-based signatures for distinguishing
shapes in two dimensions \cite{Vietal2010}, could we define shape
signatures using \MSFN{} computations to characterize the geometry of
sets in arbitrary dimensions? The sequence of optimal \MSFN{}
decompositions of a given set for varying values of the scale
parameter $\lambda$ captures all the scale information of its
geometry. Could we represent all this information in a compact manner,
for instance, in the form of a barcode?

\section*{Acknowledgments}

We acknowledge the financial support from the National Science
Foundation (NSF) through grants DMS-0914809 and CCF-1064600.

\bibliographystyle{plain} 
\bibliography{homology,flatnorm}

\end{document}